\documentclass[11pt]{amsart}

\usepackage[T1]{fontenc}
\usepackage{lmodern}
\usepackage[a4paper,inner=3cm,outer=3cm,top=4cm,bottom=4cm,pdftex]{geometry}

\usepackage{mathtools}
\usepackage{amssymb}
\usepackage{mathrsfs}
\usepackage{enumitem}

\usepackage{tikz}
\usetikzlibrary{calc,arrows.meta}

\usepackage[
  colorlinks=false,
  linkbordercolor={0 1 0},
  citebordercolor={1 0 0},
  urlbordercolor={1 0 0}
]{hyperref}

\numberwithin{equation}{section}

\newtheorem{theorem}{Theorem}[section]
\newtheorem{proposition}[theorem]{Proposition}
\newtheorem{lemma}[theorem]{Lemma}
\newtheorem{corollary}[theorem]{Corollary}

\theoremstyle{definition}
\newtheorem{definition}[theorem]{Definition}
\newtheorem{remark}[theorem]{Remark}

\newcommand{\eps}{\varepsilon}

\newcommand{\C}{\mathbb C}

\newcommand{\R}{\mathbb R}
\newcommand{\Z}{\mathbb Z}
\newcommand{\E}{\mathbb E}
\newcommand{\1}{\mathbf 1}
\renewcommand{\mod}[1]{\,(\mathrm{mod}\,#1)}
\newcommand{\Pcal}{\mathcal P}

\DeclareMathOperator{\Poly}{Poly}
\DeclareMathOperator{\TV}{TV}
\DeclareMathOperator{\vol}{vol}

\newcommand{\norm}[1]{\left\lVert #1\right\rVert}
\newcommand{\abs}[1]{\left\lvert #1\right\rvert}

\newcommand{\dd}{\,\textnormal{d}}

\renewcommand{\Re}{\operatorname{Re}}

\renewcommand{\leq}{\leqslant}
\renewcommand{\geq}{\geqslant}
\usepackage[normalem]{ulem}

\title[Higher order uniformity in shorter intervals]{Higher order uniformity of the primes and cancellation of the M\"obius function in shorter intervals}

\author{Kaisa Matom\"aki}
\address{Department of Mathematics and Statistics, University of Turku, 20014 Turku, Finland}
\email{ksmato@utu.fi}

\author{Mayank Pandey}
\address{Courant Institute of Mathematical Sciences, 251 Mercer Street, New York, NY, 10012}
\email{mayankpandey9973@gmail.com}
 
\author{Javier Pliego}
\address{Departamento de Matem\'aticas, Facultad de Ciencias, Universidad Aut\'onoma de Madrid, Ciudad Universitaria de Cantoblanco, 28049 Madrid, Spain}
\email{javier.pliego@uam.es}

\author{Joni Ter\"av\"ainen}
\address{Department of Pure Mathematics and Mathematical Statistics, University of Cambridge, Cambridge CB3 0WB, UK}
\email{joni.p.teravainen@gmail.com}

\author{Mengdi Wang}
\address{\'{E}cole Polytechnique F\'{e}d\'{e}rale de Lausanne (EPFL), Lausanne, Switzerland}
\email{mengdi.wang@epfl.ch}

\date{}

\begin{document}

\begin{abstract}
We prove the Gowers uniformity of the von Mangoldt function minus its Cram\'er model in all short intervals $[X,X+X^{3/5+\eps}]$, improving on the work of the first and fourth authors with Shao and Tao, where the exponent was $5/8$. This also implies a local-to-global theorem for linear equations in primes in such short intervals.

We also show that the M\"obius function has cancellation in all intervals $[X,X+X^{19/35+\eps}]$, improving on the work of the first and fourth authors, where the exponent was $11/20$.
  
Both improvements are based on improved treatment of trilinear sums over short intervals which naturally reduces to estimating mean values of products of three Dirichlet polynomials. In a general setting, we reduce the task of estimating mean values of products of Dirichlet polynomials with given large value bounds to the task of showing that a certain piecewise linear function is positive. This reduction allows recent large value estimates to be incorporated.

Combining this general framework with the most recent large value estimates stemming from the work of Guth and Maynard, we prove a Heath-Brown--Iwaniec type estimate for type I/II sums in intervals of length $X^{19/35+\varepsilon}$ and a Baker--Harman--Pintz parallelogram type estimate for general trilinear sums in intervals of length $X^{3/5+\varepsilon}$.
\end{abstract}

\maketitle

\makeatletter
\markboth
  {\MakeUppercase{Kaisa Matom\"aki, Mayank Pandey, Javier Pliego,
                   Joni Ter\"av\"ainen, Mengdi Wang}}
  {\MakeUppercase{Higher order uniformity in shorter intervals}}
\makeatother

\section{Introduction}
\subsection{Higher order uniformity results}
Our first main result concerns higher order uniformity of the von Mangoldt function $\Lambda$ minus its Cram\'er model in all short intervals of length $X^{3/5+\varepsilon}$. The Cram\'er model $\Lambda_X^\sharp$ is defined, following \cite[(1.1)]{MSTT}, by
\begin{equation}\label{eq:lambda-sharp}
\Lambda_X^\sharp(n)\coloneqq \frac{\Pcal(R)}{\varphi(\Pcal(R))}\,\1_{(n,\Pcal(R))=1}, \quad \text{with } R=\exp(\log^{1/10} X) \quad \text{and} \quad \mathcal{P}(R) = \prod_{p < R} p.
\end{equation}
The following theorem improves on the work of the first and fourth authors with Shao and Tao \cite[Theorem 1.5(i)]{MSTT}, where the interval length $X^{5/8+\eps}$ for the Gowers uniformity of $\Lambda-\Lambda^{\sharp}$ was obtained.

\begin{theorem}[Higher order uniformity of $\Lambda$]\label{thm:intro-uniformity}
Let $s\geq1$ be an integer and $\eps>0$. Uniformly for $X\geq3$ and
$$
X^{3/5+\eps}\leq H\leq X,
$$
one has\footnote{See Section~\ref{subsec:notation} for the definition of the Gowers norm}
\begin{equation}\label{eq:intro-gowers-theorem}
\norm{\Lambda-\Lambda_X^\sharp}_{U^s(X,X+H]}=o_{s,\eps}(1).
\end{equation}
\end{theorem}

By applying the inverse theorem for the Gowers norms, Theorem~\ref{thm:intro-uniformity} reduces to showing discorrelation of the von Mangoldt function minus its model against all fixed complexity nilsequences. Such an estimate is the key technical input for proving Theorem~\ref{thm:intro-uniformity}, and since the approach via combinatorial decompositions is quite flexible, it naturally leads to a similar result for the M\"obius function.

\begin{theorem}[Nilsequence discorrelation]\label{thm:discorrelation}
Let $d,D \in \mathbb{N}$, let $\eps>0$, $\delta\in(0,1/2)$, $X\geq 3$, and
\begin{equation*}
  X^{3/5+\eps}\leq H\leq X^{1-\varepsilon}.
\end{equation*}
Let $G/\Gamma$ be a filtered nilmanifold\footnote{The symbols relating to nilsequences have the same meanings as in \cite[Section~2.3]{MSTT}.} of degree at most $d$, dimension at most $D$, and complexity at most $1/\delta$, and let $F\colon G/\Gamma\to\C$ be a Lipschitz function of norm at most $1/\delta$. Then, for every $A>0$,
\begin{align}
  \sup_{g\in\Poly(\Z\to G)} \left\lvert\sum_{X<n\leq X+H}\mu(n)\overline{F(g(n)\Gamma)}\right\rvert^{\,*} &\ll_{A,\eps,d,D} \delta^{-O_{d,D}(1)}\frac{H}{\log^A X},\label{eq:mobius-discor}\\
  \sup_{g\in\Poly(\Z\to G)} \left\lvert\sum_{X<n\leq X+H}(\Lambda(n)-\Lambda_X^\sharp(n))\overline{F(g(n)\Gamma)}\right\rvert^{\,*} &\ll_{A,\eps,d,D} \delta^{-O_{d,D}(1)}\frac{H}{\log^A X}.\label{eq:lambda-discor}
\end{align}
\end{theorem}

\textbf{Remarks.}
\begin{enumerate}
    \item We could also obtain a similar improvement for the $k \geq 4$ case of~\cite[Theorem 1.1(iii)]{MSTT} concerning higher order uniformity of the $k$-fold divisor function $d_k$. We leave the details for the interested reader.

    \item For the von Mangoldt function, this result improves on~\cite[Theorem 1.1]{MSTT}, where one had $H\in [X^{5/8+\varepsilon}, X^{1-\eps}]$. 

    \item For the M\"obius function, this result improves on~\cite[Theorem 1.1(iv)]{MSTT}, where the length range was the same but the savings were a small power of logarithm. The smaller savings already sufficed for proving the Gowers uniformity of $\mu$ in short intervals $(X,X+X^{3/5+\varepsilon}]$ in~\cite[Theorem 1.5(iii)]{MSTT}, so we do not state this consequence here.    
\end{enumerate}

As in~\cite{MSTT}, we obtain an immediate corollary for polynomial phases.
\begin{corollary}[Polynomial phases]\label{cor:polynomial-phases}
Let $\eps>0$, $d\geq 1$, $A>0$, $X\geq 3$, and $X^{3/5+\eps}\leq H\leq X^{1-\varepsilon}$. For every polynomial $P\in\R[Y]$ of degree at most $d$,
\begin{equation}\label{eq:mu-poly}
  \left\lvert\sum_{X<n\leq X+H}\mu(n)e(P(n))\right\rvert \ll_{A,d,\eps} \frac{H}{\log^A X}.
\end{equation}
Moreover, we have
\begin{equation}\label{eq:lambda-poly}
  \left\lvert\sum_{X<n\leq X+H}\Lambda(n)e(P(n))\right\rvert \ll_{A,d,\eps} \frac{H}{\log^A X},
\end{equation}
unless, writing
\begin{equation*}
  P(\lfloor X\rfloor +m)=\alpha_0+\alpha_1m+\cdots+\alpha_dm^d,
\end{equation*}
there exists an integer $1\leq q\leq \log^{O_{A,d,\eps}(1)} X$ such that
\begin{equation}\label{eq:major-arc-poly}
  \max_{1\leq j\leq d} H^j\norm{q\alpha_j}_{\R/\Z}
  \leq \log^{O_{A,d,\eps}(1)}X.
\end{equation}
\end{corollary}

We also remark that, for the M\"obius function and when $\deg P=1$, an upcoming work of the third and fifth authors~\cite[Theorem 1.1]{PW} extends the orthogonality estimate~\eqref{eq:mu-poly} to the range $X^{0.58+\eps }\leq H \leq X$, albeit with a weaker error term $O(H(\log X)^{-1/3+\eps })$, thereby providing only a small logarithmic saving.

\begin{proof}[Proof of Corollary~\ref{cor:polynomial-phases} assuming Theorem~\ref{thm:discorrelation}]
Polynomial phases are nilsequences on the filtered torus $\R/\Z$ with filtration degree $d$, so \eqref{eq:mu-poly} follows directly from Theorem~\ref{thm:discorrelation}. For \eqref{eq:lambda-poly}, write $\Lambda=(\Lambda-\Lambda_X^\sharp)+\Lambda_X^\sharp$. The first term is covered by Theorem~\ref{thm:discorrelation}; the second is a standard major arc calculation for the Cram\'er model \eqref{eq:lambda-sharp}, giving a negligible contribution unless \eqref{eq:major-arc-poly} holds. This is the argument of \cite[Corollary~1.3]{MSTT} with $5/8$ replaced by $3/5$.
\end{proof}

\subsection{Results on short sums of arithmetic functions}

Using variants of the arithmetic inputs that go into our higher order uniformity results also leads to improvements in several results about untwisted short sums of multiplicatively defined sequences. In particular, we obtain  cancellation for sums of the M\"obius function in shorter intervals than was previously known, improving on the work of the first and fourth authors \cite[Theorem 1.1]{MatomakiTeravainenMobius}, who had the exponent $11/20 = 0.55$ in place of $19/35 \approx 0.54286$.

\begin{theorem}[M\"obius sums in short intervals]\label{thm:intro-mobius}
Let $\eps>0$. Uniformly for $X\geq3$ and
$$
  X^{19/35+\eps}\leq H\leq X,
$$
one has
\begin{equation}\label{eq:intro-mobius-theorem}
  \left|\sum_{X<n\leq X+H}\mu(n)\right|\ll_{\eps} \frac{H}{\log^{1/3-\eps} X}.
\end{equation}
\end{theorem}

 We can similarly improve other results in~\cite{MatomakiTeravainenMobius}. Firstly, we can sharpen the exponent in a result about sums of two squares in short intervals (\cite[Theorem 1.2]{MatomakiTeravainenMobius}) from $11/20$ to $19/35$.
\begin{theorem}\label{thm1.3}
Let $N_{0}(X)$ be the number of integers $\leq X$ that are expressible as sums of two squares of integers. Let $\eps > 0$ be fixed. Uniformly for $X\geq 3$ and $X^{19/35+\eps}\leq H\leq X,$ one has
$$N_{0}(X+H)-N_{0}(X)=\frac{CH}{\log^{1/2} X}+O\left(\frac{H}{\log^{2/3-\eps} X}\right), $$
where $C=\frac{1}{\sqrt{2}}\prod_{p\equiv 3\mod{4}}(1-1/p^{2})^{-1/2}$ is the Landau--Ramanujan constant.
\end{theorem}
Likewise, we can improve the exponent $\theta_{k}$ for which the expected asymptotic formula for the sum of the $k$-fold divisor functions $d_{k}$ on short intervals of length $H\geq X^{\theta_{k}+\varepsilon}$ is valid. The previous best known value for $\theta_{k}$ when $k\geq 5$ was $\theta_{k}=11/20$ (see~\cite[Theorem 1.3]{MatomakiTeravainenMobius}).
\begin{theorem}\label{thm1.4}
Let $\eps > 0$ and let $k\geq 5$ be an integer. Uniformly for $X\geq 3$ and $X^{19/35+\eps}\leq H\leq X$ one has
$$\sum_{X<n\leq X+H}d_{k}(n)=(X+H)P_{k-1}(\log(X+H))-XP_{k-1}(\log X)+O\big(H(\log X)^{(2/3+\eps)k-1}\big),$$ where $P_{k-1}$ is an explicit polynomial of degree $k-1$ (see \cite[equation (5.36) in Section II.5.4]{Tenenbaum2015})\footnote{In~\cite[Theorem 1.3]{MatomakiTeravainenMobius} the main term is simplified to $HQ_{k-1}(\log X)$ where $Q_{k-1}=P_{k-1}+P_{k-1}'$ in the current paper's notation, but in fact this is only valid if $H/X\ll(\log X)^{1-k/3+k\eps}$, and for $H$ close to $X$ one must retain the difference above.}.
\end{theorem}
Lastly, we include a result concerning $E_{2}$ almost primes, that is, numbers of the shape $p_{1}p_{2}$ with $p_{1},p_{2}$ primes, improving the exponent $11/20$ in~\cite[Theorem 1.4]{MatomakiTeravainenMobius}.
\begin{theorem}\label{thm1.5}
Let $\varepsilon>0$. Uniformly for $X\geq 100$ and $X^{19/35+\eps}\leq H\leq X$ one has
$$   \sum_{\substack{X<n\leq X+H \\ n\in E_{2}}}1=\frac{H\log\log X}{\log X}+O\Big(\frac{H\log\log\log X}{\log X}\Big). $$
\end{theorem}

%There is one final result in ~\cite{MatomakiTeravainenMobius} which deals with the M\"obius function twisted by a linear phase function in short intervals. The methods of the present paper do not lead to any further improvement in this problem, but it will be treated in an upcoming article of the third and fifth authors.

\subsection{A result on short patterns in the primes}

Theorem~\ref{thm:intro-uniformity} also yields an asymptotic local-to-global theorem for primes in short intervals, improving on \cite[Theorem~1.7]{MSTT}. The deduction of the following theorem from Theorem~\ref{thm:intro-uniformity} is the same as~\cite[Proof of Theorem 1.7 in Section~10]{MSTT} and we will not repeat it in this paper.

\begin{theorem}[Linear equations in primes in short intervals of exponent $3/5$]\label{thm:intro-linear}
Let $\eps>0$ and $X^{3/5+\eps}\leq H\leq X^{1-\eps}$. Let $d,t\geq1$ be  integers and let $L\geq1$. Let $\Psi=(\psi_1,\ldots,\psi_t)$ be a system of affine linear forms $\psi_i\colon\Z^d\to\Z$ of the shape
$$
  \psi_i(\mathbf n)=\dot\psi_i\cdot \mathbf n+\psi_i(0),
  \qquad \dot\psi_i\in\Z^d,
$$
with $0<\|\dot\psi_i\|_\infty\leq L$ and $\abs{\psi_i(0)}\leq LX$, and suppose that $\dot\psi_i$ and $\dot\psi_j$ are linearly independent over $\mathbb Q$ whenever $i\neq j$. Let $K\subset(X,X+H]^d$ be a convex body. Extend $\Lambda$ by zero to the nonpositive integers. Then,
\begin{equation}\label{eq:intro-linear-theorem}
  \sum_{\mathbf n\in K\cap\Z^d}\prod_{i=1}^t\Lambda(\psi_i(\mathbf n))
  =\beta_\infty\prod_p\beta_p+o_{d,t,L,\eps}(H^d),
\end{equation}
where
$$
  \beta_\infty=\vol\bigl(K\cap\Psi^{-1}(\R_{>0}^t)\bigr)
$$
and
$$
  \beta_p=\E_{\mathbf n\in(\Z/p\Z)^d}\prod_{i=1}^t \frac{p}{p-1}\1_{\psi_i(\mathbf n)\not\equiv0\pmod p}.
$$
\end{theorem}
\begin{remark}
Note that for all primes $p$ one has
$\beta_p=1+O_{d,t,L}(p^{-2})$. Thus the Euler product in
\eqref{eq:intro-linear-theorem} converges absolutely. It is zero if and only if
one of the local factors $\beta_p$ vanishes.
\end{remark}

\subsection{Plan for the paper}
Using Heath-Brown's identity, the work of~\cite{MSTT} on the case of non-abelian nilmanifolds, and a standard Perron formula argument, we reduce our main task to bounding mean values of products of Dirichlet polynomials.

In Section~\ref{sec:LVMVestimates} we collect some known mean and large value results for Dirichlet polynomials. In Section~\ref{sec:oneVar} we reduce the task of estimating mean values of products of Dirichlet polynomials with given large value bounds to the task of showing that a piecewise linear function is positive. This reduction is directly adaptable to other contexts.

In Sections~\ref{sec:HBI} and~\ref{sec:BHP} we utilise the results of the two previous sections to prove new mean value estimates for products of three Dirichlet polynomials. The elementary proofs that the corresponding piecewise linear functions are positive are postponed to Section~\ref{sec:Comput}.

In Section~\ref{sec:uniformity} we prove Theorem~\ref{thm:intro-uniformity} through nilsequence discorrelation, and in Section~\ref{sec:mobius-short} we prove Theorem~\ref{thm:intro-mobius}, using Ramar\'e's identity.

\subsection{Limiting cases}
In the proofs of the above theorems, there are some limiting cases for the lengths of the Dirichlet polynomials arising from Heath-Brown's identity which force one to impose the restrictions $H\geq X^{3/5+\varepsilon}$ and $H\geq X^{19/35+\varepsilon}$ respectively, preventing any further improvement (apart from possibly improving on the $+\varepsilon$, in the spirit of Heath-Brown's work~\cite{H-B-o1}) without new arithmetic information. In the setting of Theorem \ref{thm:intro-uniformity}, the particular case in which there are four Dirichlet polynomials of length $X^{1/4+o(1)}$ entails bounding the large value set of the polynomials ``corresponding to the $7/10$-line'', and this does not seem to be tractable whenever $H\leq X^{3/5}$ with the current technology available. A similar situation occurs in the context of Theorems \ref{thm:intro-mobius}--\ref{thm1.5} for the case of five Dirichlet polynomials of length $X^{1/5+o(1)}$ and the large value set of the polynomials ``corresponding to the $5/7$-line'' (see Remarks \ref{rem4.3} and \ref{rem5.3} for a more technical discussion).

\subsection{Notation}\label{subsec:notation}

We use the usual Landau and Vinogradov notation $O(\cdot), o(\cdot), \ll, \gg, \asymp$. Implied constants may depend on fixed structural parameters such as the degree of the Gowers norms, nilmanifold degree and dimension, the number of forms in a linear system, and on $\eps>0$, but not on $X$ and $H$.

As in \cite[(1.4)]{MSTT}, for an interval $I\subset\R$ of length $\geq 1$ and a function $f$ on $I\cap\Z$, define the maximal progression norm
\begin{equation}\label{eq:starred-abs}
  \left\lvert\sum_{n\in I\cap\Z}f(n)\right\rvert^{\,*} \coloneqq \sup_{P\subseteq I\cap\Z}\left\lvert\sum_{n\in P}f(n)\right\rvert,
\end{equation}
where $P$ ranges over all arithmetic progressions contained in $I\cap\Z$.

As in \cite[Definition 2.1]{MSTT}, we define the total variation norm as follows. Let $P$ be an arithmetic progression and let $f\colon P\to\mathbb C$.
Define
\[
\|f\|_{\mathrm{TV}(P)}\coloneqq\sup_{n\in P}|f(n)|+\sup_{\substack{k\geq1\\n_1<\cdots<n_k\\n_1,\ldots,n_k\in P}}\sum_{j=1}^{k-1}|f(n_{j+1})-f(n_j)|.
\]
We set $\|f\|_{\mathrm{TV}(P)}=0$ when $P=\emptyset$.
If $q\geq1$ is an integer, we also write
\[
  \|f\|_{\mathrm{TV}(P;q)}\coloneqq\sum_{a\bmod q}\|f\|_{\mathrm{TV}(P\cap(a+q\mathbb Z))}.
\]

For an integer $s\geq1$ and a finitely supported function $f\colon\Z\to\C$, define the unnormalised Gowers norm by
\begin{equation}\label{eq:gowers-on-Z}
  \norm{f}_{U^s(\Z)} \coloneqq \left(\sum_{x,h_1,\ldots,h_s\in\Z} \prod_{\omega\in\{0,1\}^s}\mathcal C^{|\omega|} f(x+\omega_1h_1+\cdots+\omega_sh_s)\right)^{1/2^s},
\end{equation}
where $\mathcal C$ denotes complex conjugation and
$|\omega|=\omega_1+\cdots+\omega_s$. For a bounded interval
$I\subset\mathbb R$ containing an integer and an arbitrary function
$f\colon\Z\to\C$, set
\begin{equation}\label{eq:local-gowers-definition}
  \norm{f}_{U^s(I)} \coloneqq \frac{\norm{f\1_I}_{U^s(\Z)}}{\norm{\1_I}_{U^s(\Z)}}.
\end{equation}
Finally, we write $e(x)=e^{2\pi i x}$ for $x\in \mathbb{R}$. 

\subsection*{Acknowledgements}

KM was supported by the Research Council of Finland grants 364214 and 370133. JP was partially supported by the project PID2024-156180NB-I00 funded by MICINN-Spain. JT was supported by the European Union's Horizon Europe research and innovation programme under ERC grant agreement no. 101162746. MW was supported by the Swiss National Science Foundation grant TMSGI2-211214.

\subsection*{AI Declaration}

ChatGPT Pro was used to draft Figure~\ref{fig:typeII-region} and to suggest the idea of simplifying the three-variable computational check to a one-variable check as we now do in Lemma~\ref{le:inversecost}. It also helped with the proofs of the computational Lemmas~\ref{le:16/35comput} and~\ref{le:2/5comput} and with proofreading.

\section{Mean and large value results for Dirichlet polynomials} \label{sec:LVMVestimates}
In this section we collect mean and large value results for Dirichlet polynomials. Let $s\in\C$. We say that a Dirichlet polynomial
\[
  A(s) := \sum_{n \in \mathbb{N}} \frac{a(n)}{n^s}
\]
has length at most $N$ if the coefficients are supported on $[1,N]$. 

\begin{definition}
Let $\nu\geq 0$ and $T \geq 1$, and let $\mathcal{T} \subseteq [T, 2T]$ be a subset. Suppose that $A(s)$ is a Dirichlet polynomial whose coefficients are supported on $(X^\beta, 2X^{\beta}]$ with $\beta > 0$. We define $\mathcal{N}_A(\nu;\mathcal{T})$ to be the maximal cardinality of a $1$-separated subset $\mathcal{U}\subseteq\mathcal{T}$ such that
\begin{equation*}
  |A(1/2+it)|\geq X^{\beta(\frac{1}{2}-\nu)}\qquad\text{for all }t\in\mathcal{U}.
\end{equation*}
\end{definition}

For example, if $|a(n)|\leq 1$, then \cite[Theorem~1.1]{GuthMaynard} implies 
\[
\mathcal N_A(\nu;[T,2T])\ll T^{o(1)} (X^{2\beta \nu}+X^{\beta(4\nu-2/5)} + T X^{\beta(4\nu-8/5)}).
\]

We will bound $\mathcal{N}_A(\nu; \mathcal{T})$ by $X^{\mathcal{K}(\beta, \nu; \tau)+o(1)}$, where the function $\mathcal{K}(\beta, \nu; \tau)$ is defined as follows.
\begin{definition}\label{def:function-K}
  We define the functions $\mathcal{K}, \mathcal{K}_0 \colon \mathbb{R}_+ \times \mathbb{R}_{\geq 0} \times (0, 1] \to \mathbb{R}$ via
  \[
\mathcal{K}_0(\beta, \nu; \tau) := \min\left\{\tau+\beta(6\nu-2), \max\left\{\tau+\beta\left(4\nu-\frac{8}{5}\right), \frac{\tau}{2}+\beta(4\nu-1)\right\}, \tau+\beta(2\nu-1)\right\},
    \]
    and
    \[
      \mathcal{K}(\beta, \nu; \tau) := \min\left\{\tau, \max\left\{2\beta \nu, \mathcal{K}_0(\beta, \nu; \tau)\right\}\right\}.
      \]
\end{definition}

To provide some intuition, we point out that the three terms in $\mathcal K_0 (\beta, \nu; \tau)$ arise from the Hal\'asz--Montgomery--Huxley large values estimates, the Chen--Gupta--Li large values estimates~\cite{CGCL} (that refine Guth--Maynard's bound), and the diagonal contribution in the second moment estimate, respectively.

\begin{lemma}[Large values estimates] \label{le:MVHuGM}
Let $B \geq 1$ be fixed. Let $\beta>0$, $X\geq3$, and
\begin{equation*}
  A(s)=\sum_{X^\beta<n\leq 2X^\beta}a(n) n^{-s},\qquad |a(n)|\leq d_B(n)\log^B X.
\end{equation*}
Let $T = X^\tau$ with $\tau \in (0, 1]$.
Then the following statements hold:

\begin{enumerate}[label=(\roman*)]
  \item For every $\nu \geq 0$ and every fixed $k \in \mathbb{N}$, we have
\begin{equation}\label{eq:hm-large-values-bound}
  \mathcal N_A(\nu ; [T, 2T])\ll X^{\mathcal{K}(\beta k,\nu;\tau)+o(1)}.
\end{equation}

  \item There is a constant $C = C(B) > 0$ such that the following holds. Let 
  \[
  a(n) = \mathbf{1}_{n \in I} \chi(n) \log^r n,
  \]
   where $I \subseteq (X^\beta, 2X^\beta]$ is an interval, $\chi$ is a Dirichlet character of modulus $\ll \log^B X$ and $r \in \{0, 1\}$. Let $\mathcal{T} \subseteq [T, 2T]$ be such that
  \begin{equation} \label{eq:4thmomcond}
|A(1/2+it)| \geq \frac{X^{\beta/2}}{T} \log^C X
\end{equation}
for every $t \in \mathcal{T}$. Then, if $\tau>1/\sqrt{\log X}$, for every $\nu \geq 0$ we have
  \begin{equation} \label{eq:fourthMomentN}
    \mathcal{N}_A(\nu; \mathcal{T}) \ll X^{\tau + 2\beta(2\nu-1) +o(1)}.
  \end{equation}
\end{enumerate}
\end{lemma}

\begin{proof}
  \textbf{Proof of (i).} We first note that it suffices to prove
  \begin{equation} \label{eq:LVk=1}
    \mathcal{N}_A(\nu; [T, 2T])
    \ll X^{\mathcal{K}(\beta, \nu; \tau) +o(1)}.
  \end{equation}
  Indeed, let $k\in\mathbb N$ be fixed and write
  \[
    A(s)^k=\sum_{n\in \mathbb{N}}\frac{c_k(n)}{n^s}.
  \]
  The coefficients are supported on $N^k<n\leq2^kN^k$, where
  $N=X^\beta$, and satisfy
  \[
    |c_k(n)|\leq d_{Bk}(n) \log^{Bk} X.
  \]
  Split
  \[
    A(s)^k=\sum_{\ell=0}^{k-1}A_{k,\ell}(s),\qquad  A_{k,\ell}(s)=\sum_{2^\ell N^k<n\leq2^{\ell+1}N^k}\frac{c_k(n)}{n^s}.
  \]
  Let $\mathcal U\subseteq[T,2T]$ be a $1$-separated set on which
  \[
    |A(1/2+it)|\geq X^{\beta(1/2-\nu)}.
  \]
  Then, for every $t\in\mathcal U$, at least one
  $\ell\in\{0,\ldots,k-1\}$ satisfies
\[
|A_{k,\ell}(1/2+it)|\geq \frac1k X^{k\beta(1/2-\nu)}.
\]
  We partition $\mathcal U$ according to the least such $\ell$. Since
  $2^{\ell+1}N^k=X^{k\beta+O_k(1/\log X)}$, applying
  \eqref{eq:LVk=1} to each $A_{k,\ell}$, with the parameters changed
  by $O_{B,k,\beta,\nu}(1/\log X)$, gives
\[
    \#\mathcal U\ll_{B,k}X^{\mathcal K(k\beta,\nu;\tau)+o(1)},
\]
  where we used that $\mathcal K$ is piecewise linear. This proves the
  reduction to \eqref{eq:LVk=1}.

     We now turn to the proof of \eqref{eq:LVk=1}. Recall that $T = X^\tau$. If $\mathcal{K}(\beta, \nu; \tau) = \tau$, then $X^{\mathcal{K}(\beta, \nu; \tau)} = X^\tau = T$ and the claim is trivial. We may therefore assume henceforth that
      \begin{equation} \label{eq:Ksimp}
         \mathcal{K}(\beta, \nu; \tau) = \max\left\{2\beta \nu, \mathcal{K}_0(\beta, \nu; \tau)\right\} \leq \tau.
       \end{equation}
       
Set $N=X^\beta$ and $V=N^{1/2-\nu}$, and let
$\mathcal U\subseteq[T,2T]$ be any $1$-separated set such that
\[
  |A(1/2+it)|\geq V\qquad\text{for }t\in\mathcal U.
\]
Write
$$
  G\coloneqq \sum_{N<n\leq2N}\frac{|a(n)|^2}{n}\ll \log^{2B} X N^{-1}\sum_{N<n \leq 2N} d_B(n)^2\ll \log^{O_B(1)} X.
$$
By the second moment large values theorem \cite[Theorem~9.7]{IwaniecKowalski}, one has
\begin{equation}\label{eq:IK-thm-97-use}
  \#\mathcal{U}\ll (N+T)GV^{-2}\log(2N) \ll (X^{2\beta\nu}+X^{\tau+\beta(2\nu-1)})X^{o(1)}.
\end{equation}
By Huxley's large values estimate in the form of \cite[Corollary~9.9]{IwaniecKowalski}, one has
\begin{equation}\label{eq:huxley-use}
  \#\mathcal{U}\ll \left(GNV^{-2}+G^3N TV^{-6}\right) \log^6 (2T) \ll (X^{2\beta \nu}+X^{\tau+\beta(6\nu-2)})X^{o(1)}.
\end{equation}

Next we apply the large value results of Guth and
Maynard~\cite{GuthMaynard}, together with the refinement due to Chen,
Gupta, and Li~\cite{CGCL}. Let $C_B$ be sufficiently large and set
\[
  A^\ast(it)\coloneqq N^{1/2-C_B/\log\log N}A(1/2+it)=\sum_{N<n\leq2N}b(n)n^{-it},
\]
where
\[
  b(n)=\frac{a(n)}{N^{C_B/\log\log N}}\left(\frac Nn\right)^{1/2}.
\]
The standard bound for the divisor function shows that if $C_B$ is large enough, then $|b(n)|\leq1$ for all $n$. Moreover, for
$t\in\mathcal U$,
\[
  |A^\ast(it)|\geq V^\ast,\qquad V^\ast\coloneqq N^{1-\nu-C_B/\log\log N}.
\]

Suppose first that $\beta\geq5\tau/6$. Then, for all sufficiently
large $X$,
\[
  N=X^\beta\geq(2T)^{2/3}.
\]
We may therefore apply
\cite[Theorem~1.1]{CGCL}, with $q=q_1=1$ and height $2T$, to the
points $-t$ with $t\in\mathcal U$. It gives
\[
  \#\mathcal U\ll (2T)^{o(1)}\left(N^2(V^\ast)^{-2}+(2T)^{1/2}N^3(V^\ast)^{-4}+2TN^2(V^\ast)^{-4}+2TN^{12/5}(V^\ast)^{-4}
  \right).
\]
The third term is dominated by the fourth. Since
$V^\ast=N^{1-\nu+o(1)}$, we obtain
\begin{equation} \label{eq:TboundGMbetasmall}
  \# \mathcal{U}\ll\left(X^{2\beta\nu}+X^{\tau/2+\beta(4\nu-1)}+X^{\tau+4\beta\nu-8\beta/5}\right)X^{o(1)}.
\end{equation}
If instead $\beta \leq 5\tau/6$, we have from \cite[Theorem 1.1]{GuthMaynard} with $V=N^{1-\nu+o(1)}$ that
\[
\# \mathcal U\ll \big(X^{2\beta \nu}+ X^{\tau +\beta(4\nu-8/5)} + X^{\beta(4\nu-2/5)}\big) X^{o(1)}.
\]
Noting that $\beta (4\nu-2/5) \leq \tau/2 +4\beta \nu -\beta$ whenever $\beta \leq 5\tau/6$, we can also conclude estimate~\eqref{eq:TboundGMbetasmall} in this case.
 
Now~\eqref{eq:LVk=1} follows by comparing the estimates~\eqref{eq:IK-thm-97-use},~\eqref{eq:huxley-use}, and~\eqref{eq:TboundGMbetasmall} with~\eqref{eq:Ksimp} and the definition of $\mathcal{K}_0(\beta, \nu; \tau)$. 

\textbf{Proof of~(ii).} Now we assume that $a(n) = \mathbf{1}_{n \in I} \chi(n) \log^r n$ for some Dirichlet character $\chi$ of modulus $\ll \log^{O(1)} X$, some interval $I \subseteq (X^\beta, 2X^\beta]$, and some $r \in \{0, 1\}$.

Write $N=X^\beta$. Perron's formula and the fourth moment estimate for Dirichlet $L$-functions give that there is a constant
$C_0=C_0(B)>0$ such that, for every $1$-separated set
$\mathcal U\subseteq[T,2T]$,
\begin{equation}\label{eq:fourth-moment-sublemma}
 \sum_{t\in\mathcal U}|A(1/2+it)|^4
 \ll \left(T+\#\mathcal U\left(\frac{1}{T^2}+\frac{N^2}{T^4}\right)
 \right)\log^{C_0} X;
\end{equation}
this follows from the proof of
\cite[Lemma~9]{BakerHarmanPintz1997} and partial summation to remove
the coefficients $\log^r n$.

Let $\nu\geq0$ be arbitrary. Let $\mathcal U$ be a $1$-separated subset of
\[
 \left\{t\in\mathcal T:
 |A(1/2+it)|\geq N^{1/2-\nu}\right\}.
\]
We choose the constant $C=C(B)$ in \eqref{eq:4thmomcond} so that
$4C>C_0+10$. If $\nu\geq1/2$, the trivial bound
$\#\mathcal U\ll T$ for a $1$-separated subset of $[T,2T]$ gives
\[
 \mathcal N_A(\nu;\mathcal T)
 \ll T
 \leq X^{\tau+2\beta(2\nu-1)},
\]
so we may suppose that $\nu<1/2$.

The definition of $\mathcal U$ and
\eqref{eq:fourth-moment-sublemma} give
\begin{align}
 \#\mathcal U
 &\ll
 N^{4\nu-2}T\log^{C_0} X +
 \#\mathcal U
 \left(
 N^{4\nu-2}T^{-2}+N^{4\nu}T^{-4}
 \right)\log^{C_0} X.
 \label{eq:fourth-moment-absorption}
\end{align}
Since $\nu<1/2$ and $\tau > 1/\sqrt{\log X}$, we have
\begin{equation} \label{eq:fthtriv}
 N^{4\nu-2}T^{-2}\log^{C_0} X=o(1).
\end{equation}

Now, let $\nu_1$ be defined by
\[
 N^{1/2-\nu_1}=\frac{N^{1/2}}{T}\log^C X.
\]
The hypothesis $\tau>1/\sqrt{\log X}$ gives, for all sufficiently
large $X$,
\[
 \nu_1
 =\frac{\tau\log X-C\log\log X}{\beta\log X}
 >0.
\]
If $\nu_1\leq\nu$, assumption~\eqref{eq:4thmomcond} shows that every
$t\in\mathcal T$ belongs to the level sets corresponding to both
$\nu_1$ and $\nu$. Hence
\[
 \mathcal N_A(\nu;\mathcal T)
 =
 \mathcal N_A(\nu_1;\mathcal T),
\]
and it is enough to prove~\eqref{eq:fourthMomentN} with $\nu_1$ in
place of $\nu$. We may therefore replace $\nu$ by $\nu_1$ in this
case and hence assume that $\nu\leq\nu_1$. It follows that
\[
 N^{4\nu}T^{-4}\leq N^{4\nu_1}T^{-4}= \frac{1}{\log^{4C} X}.
\]
Combining with~\eqref{eq:fourth-moment-absorption} and~\eqref{eq:fthtriv} we obtain
\[
 \#\mathcal U
 \ll
 N^{4\nu-2}T\log^{C_0} X + o(\#\mathcal U)
 =
 X^{\tau+2\beta(2\nu-1)+o(1)} + o(\#\mathcal U),
\]
and~\eqref{eq:fourthMomentN} follows. This completes the proof of the lemma.
\end{proof}

We shall also require the following second moment estimate for sparse sets in Section 3.

\begin{lemma}[Sparse mean square]\label{lem:sparse-mean-square}
Let $N, T \geq 2$ and
$$
  A(s)=\sum_{N<n\leq2N}a(n)n^{-s}
$$
with $a(n) \in \mathbb{C}$. Let $\mathcal{U}\subset[-T,T]$ be $1$-separated. Then, for $\eta \in (0, 1/2]$,
\begin{equation*}
  \sum_{t\in\mathcal{U}}|A(it)|^2
  \ll_\eta
  \left(N+|\mathcal{U}| \cdot T^{\frac{9}{2} \cdot \eta^{3/2}} N^{1-\eta}\log^{2/3} T\right) \sum_{N < n \leq 2N} |a(n)|^2.
\end{equation*}
\end{lemma}

\begin{proof}
  The proof is a slight variant of the proof of the Hal\'asz--Montgomery mean value theorem (see e.g.~\cite[Theorem 9.6]{IwaniecKowalski}). For a similar variant for prime-supported Dirichlet polynomials, see~\cite[Lemma 4.4]{MRII}.

  By the duality principle (see e.g.~\cite[Section 7.1]{IwaniecKowalski}), it suffices to show that, for any complex coefficients $c(t)$, one has
\begin{equation}\label{eq:sparse-mean-square-discrete2}
\sum_{N < n \leq 2N}\left| \sum_{t\in\mathcal{U}}c(t) n^{-it}\right|^2 \ll \left(N+|\mathcal{U}| T^{\frac{9}{2}\eta^{3/2}} N^{1-\eta}\log^{2/3} T\right) \sum_{t \in \mathcal{U}} |c(t)|^2.
\end{equation}
Let $w$ be a smooth non-negative compactly supported function such that $w(x) = 1$ for $x \in [1,2]$. Its Mellin transform
\[
  \widetilde{w}(z)=\int_0^\infty w(x)x^{z-1}\dd x
\]
satisfies, for $1/2\leq\sigma\leq2$,
\begin{equation}\label{eq:mellin-decay-sparse}
  \widetilde w(\sigma+it)\ll(1+|t|)^{-10}.
\end{equation}

By inserting this smooth weight $w$, we have
\begin{align} \label{eq:HM1ststep}
\sum_{N < n \leq 2N}\left| \sum_{t\in\mathcal{U}}c(t) n^{-it}\right|^2 \leq \sum_{n} w\left(\frac{n}{N} \right)\left| \sum_{t\in\mathcal{U}}c(t) n^{-it}\right|^2 = \sum_{t_1, t_2 \in \mathcal{U}} c(t_1)\overline{c(t_2)} \sum_{n} w\left(\frac{n}{N} \right) n^{-i(t_1-t_2)}.
\end{align}
For the inner summation, Mellin inversion on a line $c>1$ gives
\[
w\left(\frac{n}{N} \right) =\frac{1}{2\pi i} \int_{c-i\infty}^{c+i\infty} \widetilde w (z) N^zn^{-z} \dd z.
\]
Hence
\[
  \sum_{n} w\left(\frac{n}{N} \right) n^{-i(t_1-t_2)} =\frac1{2\pi i}\int_{c-i\infty}^{c+i\infty}\widetilde w(z)N^z\zeta(z+i(t_1-t_2))\dd z.
\]
Moving the contour to $\Re(z)=1-\eta$ crosses a pole at $z=1+i(t_2-t_1)$. This gives
\[
 \sum_{n} w\left(\frac{n}{N} \right) n^{-i(t_1-t_2)} = \widetilde w(1+i (t_2-t_1)) N^{1+ i(t_2-t_1)} +\frac{1}{2\pi i} \int_{1-\eta-i\infty}^{1-\eta+i\infty} \widetilde w(z)N^z\zeta(z+i(t_1-t_2))\dd z.
\]
We bound the integral using \eqref{eq:mellin-decay-sparse} and a result of Ford~\cite[Theorem 1]{Ford02}, which for $|t|\geq3$ and $\sigma \in [1/2, 1]$ gives
\[
  |\zeta(\sigma+it)| \leq 100 (|t|+1)^{\frac{9}{2}(1-\sigma)^{3/2}} \log^{2/3} (|t|+2).
  \]
For $|t|\leq3$, we instead use
$|\zeta(1-\eta+it)|\ll_\eta 1$. This leads to
\begin{align*}
 &\left|
 \int_{1-\eta-i\infty}^{1-\eta+i\infty}
 \widetilde w(z)N^z\zeta(z+i(t_1-t_2))\dd z
 \right|\\
 &\leq N^{1-\eta}\int_{-\infty}^{\infty}
 |\widetilde w(1-\eta+iu)|
 |\zeta(1-\eta+i(u+t_1-t_2))|\,\dd u\\
 &\ll_\eta N^{1-\eta}\int_{-\infty}^{\infty}
 \frac{(|u+t_1-t_2|+1)^{\frac92\eta^{3/2}}
 \log^{2/3}(|u+t_1-t_2|+2)}
 {(1+|u|)^{10}}\,\dd u\\
 &\ll_\eta N^{1-\eta}
 (1+|t_1|+|t_2|)^{\frac92\eta^{3/2}}
 \log^{2/3}(2+|t_1|+|t_2|).
\end{align*}
Thus, one concludes from~\eqref{eq:HM1ststep} and the above estimates that
\begin{align*}
\sum_{N < n \leq 2N}\left| \sum_{t\in\mathcal{U}}c(t) n^{-it}\right|^2 &\ll N \sum_{t_1,t_2\in \mathcal U} |c(t_1) \overline{c(t_2)}| \cdot|\widetilde w(1+i (t_2-t_1))| \\
&+ N^{1-\eta} \sum_{t_1,t_2\in \mathcal U} |c(t_1)\overline{c(t_2)}| (1+|t_1|+|t_2|)^{\frac{9}{2}\eta^{3/2}} \log^{2/3}(2+|t_1|+|t_2|).
\end{align*}
The claim~\eqref{eq:sparse-mean-square-discrete2} now follows from the inequality $|c(t_1)\overline{c(t_2)}| \leq (|c(t_1)|^2+|c(t_2)|^2)/2$, the rapid decay of $\widetilde{w}$ (see~\eqref{eq:mellin-decay-sparse}) and the assumption that $\mathcal{U}$ is $1$-separated.
\end{proof}

\section{Bounding Dirichlet polynomial mean values via a one-variable check} \label{sec:oneVar}
The following simple lemma allows us to reduce a multivariable check to a one-variable check. There and later we use the convention that the infimum of an empty set is $\infty$.
\begin{lemma} \label{le:inversecost}
 Let $\tau>0$, $\varepsilon_0 >0$, $J\in \mathbb N$ and $I \subseteq \mathbb{R}_{\geq 0}$. For each $j \in \{1, \dotsc, J\}$, let $\beta_j \geq 0$, let $H_j \colon I \to [0, \tau],$ and define
  \[
    P_j(t) = \inf\{\beta_j y \colon y \in I, H_j(y) \geq t\}.
  \]
  Assume that
  \begin{equation} \label{eq:SumPjbound}
    \sum_{j = 1}^J P_j(t) > t+\varepsilon_0 \qquad \text{ for all } t\in[0,\tau].
  \end{equation}
  Then, for arbitrary $y_1, \dotsc, y_J \in I$ one has
  \begin{equation} \label{eq:minHjbound}
    \min_{j \in \{1, \dotsc, J\}} H_j(y_j) < \sum_{j = 1}^J \beta_j y_j -\varepsilon_0.
  \end{equation}
  \end{lemma}

\begin{proof}
We argue by contradiction and assume that~\eqref{eq:minHjbound} fails. Thus there exist $y_1, \dotsc, y_J \in I$ and $t_0 \in \mathbb{R}$ such that
  \[
     \min_{j \in \{1, \dotsc, J\}} H_j(y_j) \geq t_0 \geq \sum_{j = 1}^J \beta_j y_j -\varepsilon_0.
\]
    By the range of $H_j$, we can assume that $t_0 \in [0, \tau].$ For each $j \in \{1, \dotsc, J\}$, by the definition of $P_j(t)$, the inequality $H_j(y_j) \geq t_0$ implies that $P_j(t_0) \leq \beta_j y_j$. Thus
    \[
     t_0 \geq \sum_{j = 1}^J \beta_j y_j -\varepsilon_0 \geq \sum_{j=1}^J P_j(t_0) - \varepsilon_0.
    \]
    Consequently, the assumption~\eqref{eq:SumPjbound} fails, and the claim follows.
  \end{proof}

The following proposition reduces the task of bounding a mean value of a product of Dirichlet polynomials using large value estimates to a one-variable check. In the following two sections we shall prove two applications of Proposition~\ref{prop:one-var-check} and the estimates in Section~\ref{sec:LVMVestimates}. These applications are tailored for our needs, and one could similarly prove variants of these results for other needs.

  \begin{proposition} \label{prop:one-var-check} 
  Let $\varepsilon > 0$ be small but fixed and let $B \geq 1$. Let $X \geq 2,\, T = X^\tau$ with $\tau \in (0, 1/2)$ and let $\mathcal{T} \subseteq [T, 2T]$. Let $J \geq 3$ be a fixed integer and let $\beta_j \geq \varepsilon$ for $j = 1, \dotsc, J$ be numbers satisfying $\beta_1 + \dotsb + \beta_J = 1$. For $1\leq j\leq J$, let
$$
  A_j(s)=\sum_{X^{\beta_j} <n\leq2X^{\beta_j}} a_{j}(n) n^{-s},
$$
be a Dirichlet polynomial for which $a_j(n) \ll d_B(n) \log^B X$ for all $n$. Suppose that there exists $j_0 \in \{1, \dotsc, J\}$ such that
  \[
   \sup_{t \in \mathcal{T}} \big|A_{j_0}(1/2+it) \big| \ll_C \frac{X^{\beta_{j_0}/2}}{\log^C X}
    \]
for every $C \geq 1.$

For each $1\leq j\leq J$, suppose that the function $H_j\colon\mathbb R_{\geq0}\to[0,\tau]$ is such that
\begin{equation} \label{eq:NAjbound}
  \mathcal N_{A_j}(\nu;\mathcal T)
  \leq X^{H_j(\nu)+\varepsilon^{20}}
\end{equation}
for all $\nu\in[\varepsilon^6,\varepsilon^{-1}]$.
Define the functions $P_j\colon[0,\tau]\to[0,\infty]$ via
\[
P_{j}(t) = \inf\{\beta_j \nu \colon \nu \in [\varepsilon^6, \varepsilon^{-1}], H_j(\nu) \geq t\}.
\]
Assume that
\[
\sum_{j = 1}^J P_j(t) > t+\varepsilon^{10} \qquad\text{ for every } t \in [0, \tau].
 \]
Then
\begin{equation*}
\int_{\mathcal{T}} |A_1(1/2+it)A_2(1/2+it) \dotsm A_J(1/2+it)|\dd t\ll_{\varepsilon, A, B} \frac{X^{1/2}}{\log^A X}
\end{equation*}
for every $A>0$.
\end{proposition}

\begin{remark}
  By Lemma~\ref{le:MVHuGM}, for $\nu \in [\varepsilon^6,\varepsilon^{-1}]$ we may take
  \[
    H_j(\nu)
    =
    \min_{1\leq k\leq K_\varepsilon}
    \mathcal K(k\beta_j,\nu;\tau),
    \qquad
    K_\varepsilon=\lceil\varepsilon^{-7}\rceil.
  \]
 Note that using larger $k$ would not be helpful. Indeed, if $k>K_\varepsilon$, then
  $2k\beta_j\nu\geq2>\tau$, and hence
  $\mathcal K(k\beta_j,\nu;\tau)=\tau$, so such $k$ cannot decrease the
  displayed minimum.
\end{remark}

\begin{proof}
  By a standard reduction, it suffices to show that, for any $1$-separated subset $\mathcal{U} \subseteq \mathcal{T}$ we have, for any $A \geq 1$,
\begin{align}\label{eq:Sdef}
   \sum_{t\in\mathcal U}
  |A_1(1/2+it)\dotsm A_J(1/2+it)| \ll_{\varepsilon, A, B} \frac{X^{1/2}}{\log^A X} 
\end{align}
  We split $\mathcal{U} = \mathcal{U}_1 \sqcup \mathcal{U}_2 \sqcup \mathcal{U}_3$, where $\mathcal{U}_1$ consists of those $t \in \mathcal{U}$ for which there exists $j_1 \in \{1, \dotsc, J\}$ such that $|A_{j_1}(1/2+it)| \leq X^{-1/2}$, $\mathcal{U}_2$ consists of those $t \in \mathcal{U} \setminus \mathcal{U}_1$ for which there exists $j_2 \in \{1, \dotsc, J\}$ such that $|A_{j_2}(1/2+it)| \geq X^{\beta_{j_2}(1/2-\varepsilon^6)}$, and finally, we set $\mathcal{U}_3 = \mathcal{U} \setminus (\mathcal{U}_1 \cup \mathcal{U}_2)$. It suffices to show that, for each $k \in \{1, 2, 3\}$, we have
  \[
    S_k := \sum_{t \in \mathcal{U}_k} |A_1(1/2+it)A_2(1/2+it) \dotsm A_J(1/2+it)| \ll_{\eps, A, B} \frac{X^{1/2}}{\log^A X} 
    \]
    for every $A \geq 1$.

    \textbf{Estimation of $S_1.$} For $j \in \{1, \dotsc, J\}$, let $\mathcal U_{1,j}$ consist of those $t\in\mathcal U_1$ for which $j$ is the least index satisfying
\[
  |A_j(1/2+it)|\leq X^{-1/2}.
\]
Then $\mathcal U_1=\bigsqcup_{j=1}^J\mathcal U_{1,j}$.
For $t \in \mathcal{U}_{1, j}$ we use for $\ell \neq j$ the crude estimate
    \[
    |A_\ell(1/2+it)| \leq \sum_{X^{\beta_\ell}< n \leq 2X^{\beta_\ell}} \frac{|a_\ell(n)|}{n^{1/2}} \ll X^{\beta_\ell/2+o(1)}.
    \]
    Combining this estimate with the fact that $\beta_1+\cdots+\beta_J=1$ and $\beta_j\geq \eps$, we then deduce that the contribution of $\mathcal{U}_{1, j}$ to $S_1$ is
  \[
\#\mathcal{U}_{1, j} \cdot \frac{1}{X^{1/2}} \cdot X^{(1-\beta_{j})/2+o(1)} \leq \#\mathcal{U}_{1, j} \cdot X^{-\varepsilon/2+o(1)}.
    \]
Since $\#\mathcal{U} \ll X^\tau \leq X^{1/2}$, we obtain, for all sufficiently large $X$ and any $A \geq 1$, that
\[
  S_1\ll \#\mathcal{U}_{1} \cdot X^{-\varepsilon/2+o(1)} \ll_{\eps, A} \frac{X^{1/2}}{\log^A X}.
\]
    
    \textbf{Estimation of $S_2$.} Let us first bound $\# \mathcal{U}_{2}$. For $j \in \{1, \dotsc, J\}$, let $\mathcal U_{2,j}$ consist of those $t\in\mathcal U_2$ for which $j$ is the least index satisfying
\[
  |A_j(1/2+it)|\geq X^{\beta_j(1/2-\varepsilon^6)}.
\]
Then $\mathcal U_2=\bigsqcup_{j=1}^J\mathcal U_{2,j}$. Let $j \in \{1, \dotsc, J\}$. Choose $k_j\in\mathbb N$ so that $k_j\beta_j\in[1,2]$. Lemma~\ref{le:MVHuGM}, applied to $\mathcal U_{2,j}$ with $k=k_j$ and $\nu=\varepsilon^6$, gives
\[
  \#\mathcal U_{2,j}
  \ll X^{\min\{\tau,\max\{2k_j\beta_j\varepsilon^6,
       \mathcal K_0(k_j\beta_j,\varepsilon^6;\tau)\}\}+o(1)}.
\]
Since
\[
\mathcal K_0(k_j\beta_j,\varepsilon^6;\tau)\leq \tau+k_j\beta_j(2\varepsilon^6-1)\leq \tau+2\varepsilon^6-1<0,
\]
we obtain
\[
\#\mathcal U_2 \ll \sum_{j=1}^J X^{2k_j\beta_j\varepsilon^6+o(1)}\ll_\varepsilon X^{4\varepsilon^6+o(1)}.
\]

To estimate $S_2$, we next partition $\{1,\dotsc,J\}=\{j_0\}\sqcup\mathcal J_1\sqcup\mathcal J_2$, where $\mathcal J_1$ and $\mathcal J_2$ are non-empty, and, for $k\in\{1,2\}$, set
\[
  M_k=\prod_{j\in\mathcal J_k}X^{\beta_j},\qquad
  B_k(s)=\prod_{j\in\mathcal J_k}A_j(s)
       =\sum_{M_k<n\leq2^{|\mathcal J_k|}M_k}\frac{b_k(n)}{n^s}.
\]
Splitting each $B_k(s)$ into $O_\varepsilon(1)$ dyadic Dirichlet polynomials and applying Lemma~\ref{lem:sparse-mean-square} with $\eta=\varepsilon^4$ to each of them, we obtain 
\begin{align*}
  \sum_{t\in\mathcal U_2}|B_k(1/2+it)|^2
  &\ll\left(M_k+|\mathcal U_2|T^{\frac92\varepsilon^6}
       M_k^{1-\varepsilon^4}\right)\log^{O_{B,\varepsilon}(1)} X\\
  &\ll M_k\left(1+X^{10\varepsilon^6-\varepsilon^5}\right)
       \log^{O_{B,\varepsilon}(1)} X
  \ll M_k\log^{O_{B,\varepsilon}(1)} X
\end{align*}
once $\varepsilon$ is sufficiently small. Thus, by the pointwise bound for $A_{j_0}(1/2+it)$ and the Cauchy--Schwarz inequality, we obtain, for any $C \geq 1$,
      \[
        S_2 \ll_{B, C} \frac{X^{\beta_{j_0}/2}}{\log^C X} (M_1M_2)^{1/2} \log^{O(1)} X.
       \]
This is $\ll X^{1/2}/\log^A X$ for any $A \geq 1$.

\textbf{Estimation of $S_3$.} We split into level sets according to the sizes of $|A_j(1/2+it)|$.

For $j \in \{1, \dotsc, J\}$, we define
\[
  \nu_{j,\ell}=\min\left\{
    \varepsilon^{-1},\varepsilon^6+\frac{\ell\log2}{\beta_j\log X}\right\},\qquad 0\leq\ell\leq L_j
\]
with
\[
  L_j=\left\lceil\frac{\beta_j(\varepsilon^{-1}-\varepsilon^6)\log X}{\log2}\right\rceil.
\]

Now
\[
\beta_j\left(\frac{1}{2}-\nu_{j, 0}\right) = \beta_j\left(\frac{1}{2}-\varepsilon^6\right), \quad \beta_j\left(\frac{1}{2}-\nu_{j, L_j}\right) \leq \beta_j\left(\frac{1}{2}-\varepsilon^{-1}\right) \leq -\frac{1}{2}.
\]
Thus, for every $t\in\mathcal U_3$ and every $j$, there is an
$\ell_j\in\{0,\ldots,L_j-1\}$ such that
\[
 X^{\beta_j(1/2-\nu_{j,\ell_j+1})}<|A_j(1/2+it)|\leq X^{\beta_j(1/2-\nu_{j,\ell_j})}.
\]

Partitioning $\mathcal U_3$ according to
$(\ell_1,\ldots,\ell_J)$, and using
\[
  0\leq \beta_j(\nu_{j,\ell_j+1}-\nu_{j,\ell_j})\leq\frac{\log2}{\log X},
\]
we obtain
\[
 S_3\ll 2^J \log^J X \max_{\ell_1,\ldots,\ell_J}X^{\sum_{j=1}^J\beta_j(\frac{1}{2}-\nu_{j,\ell_j+1})}\min_j \mathcal{N}_{A_j}(\nu_{j, \ell_j+1}; \mathcal{U}_3).
\]
In particular by assumption~\eqref{eq:NAjbound}, we have $S_3 \ll X^{1/2} \log^{-A} X$ if
\[
\min_{j \in \{1, \dotsc, J\}} H_j(\nu_j) \leq \beta_1 \nu_1 +\dotsb + \beta_J \nu_J- \varepsilon^{10} \quad \text{for every $\nu_1, \dotsc, \nu_J \in [\varepsilon^6,\varepsilon^{-1}]$}.
\]
By Lemma~\ref{le:inversecost}, this happens if, for every $t \in [0, \tau]$, we have
\[
\sum_{j = 1}^J P_j(t) > t + \varepsilon^{10},
\]
    and the claim follows.
  \end{proof}

\section{A Heath-Brown--Iwaniec type trilinear estimate} \label{sec:HBI}

We now prove a trilinear estimate for Dirichlet polynomials in the spirit of the work of Heath-Brown and Iwaniec \cite[Lemma 2]{HeathBrownIwaniec1979}.

\begin{lemma}[Heath-Brown--Iwaniec trilinear estimates]\label{lem:hbi-trilinear-estimates}
  Let $\varepsilon > 0$ and $B \geq 1$. Set $T_0\coloneqq \exp(\log^{1/4} X)$ and let
$$
T_1=X^{\tau_1},\qquad \tau_1 := \frac{16}{35}-\varepsilon.
$$
Let $\beta_1, \beta_2, \beta_3 \in [\varepsilon, \tau_1]$ be such that $\beta_1 + \beta_2 + \beta_3 = 1$. For $j \in \{1, 2, 3\}$, let
$$
A_j(s)=\sum_{X^{\beta_j} <n\leq2X^{\beta_j}} a_{j}(n) n^{-s},
$$
where $|a_j(n)|\ll d_B(n)\log^B X$. Assume further that
$a_1(n)=\mathbf1_{n\in I}\chi(n)\log^r n$ for some interval
$I\subseteq(X^{\beta_1},2X^{\beta_1}]$, some $r\in\{0,1\}$,
and some Dirichlet character $\chi$ of modulus $\ll\log^B X$.

Then
\begin{equation}\label{eq:hbi-trilinear-integral}
  \int_{T_0}^{T_1} |A_1(1/2+it)A_2(1/2+it)A_3(1/2+it)|\dd t \ll_{A, B}\frac{X^{1/2}}{\log^A X}
\end{equation}
for every $A>0$.
\end{lemma}

An important computational step in the proof of Lemma~\ref{lem:hbi-trilinear-estimates} is the following lemma which we will prove in Section~\ref{sec:Comput}.
\begin{lemma} \label{le:16/35comput}
Let $\varepsilon \in (0, 1/100), \, \tau_1=\frac{16}{35}-\varepsilon$ and $ \tau \in [\varepsilon, \tau_1]$. Let $\beta_1, \beta_2, \beta_3 \in [\varepsilon, \tau_1]$ with $\beta_1+\beta_2+\beta_3=1$. For $\nu \geq 0$, define
  \begin{align*}
    H_{1,\tau}(\nu) &:= \max\left\{0,\min\{\tau + 2\beta_1(2\nu-1), \mathcal{K}(2\beta_1, \nu; \tau)\}\right\} \\
    H_{2,\tau}(\nu) &:= \min\{\mathcal{K}(\beta_2, \nu; \tau), \mathcal{K}(2\beta_2, \nu; \tau)\} \\
    H_{3,\tau}(\nu) &:= \min\{\mathcal{K}(\beta_3, \nu; \tau), \mathcal{K}(2\beta_3, \nu; \tau)\}.
  \end{align*}
For $j \in \{1, 2, 3\}$, define the function $P_{j,\tau} \colon [0, \tau] \to \mathbb{R}_{\geq 0}$ via
  \[
    P_{j,\tau}(t) = \inf\{\beta_j \nu \colon \nu \in [\varepsilon^6, \varepsilon^{-1}], H_{j,\tau}(\nu) \geq t\}.
  \]

Then, for every $t \in [0, \tau]$, we have
 \begin{equation} \label{eq:HBICompPjClaim}
   P_{1,\tau}(t)+P_{2,\tau}(t)+P_{3,\tau}(t) > t+\varepsilon^{10}.
 \end{equation}
\end{lemma}

\begin{proof}[Proof of Lemma~\ref{lem:hbi-trilinear-estimates} assuming Lemma~\ref{le:16/35comput}]
  Notice first that by our assumption on the coefficients $a_1(n)$ (use e.g.~\cite[Lemma~3.9(i)]{MSTT} and partial summation), we have for every $A > 0$ the bound
\[
\sup_{T_0 \leq t \leq T_1} |A_1(1/2+it)| \ll \frac{X^{\beta_1/2}}{\log^A X}.
\]
Combining the preceding bound with the Cauchy--Schwarz inequality and the mean value theorem gives, for every $A > 0$,
  \begin{align*}
    &\int_{T_0}^{X^\varepsilon} |A_1(1/2+it) A_2(1/2+it) A_3(1/2+it)| \dd t \\
    &\ll_A \frac{X^{\beta_1/2}}{\log^A X} \left(\int_{T_0}^{X^\varepsilon} |A_2(1/2+it)|^2 \dd t \int_{T_0}^{X^\varepsilon} |A_3(1/2+it)|^2 \dd t\right)^{1/2} \\
    &\ll_{A, B} \frac{X^{\beta_1/2}}{\log^A X} \left((X^\varepsilon+X^{\beta_2})(X^\varepsilon+X^{\beta_3})\right)^{1/2} \log^{O_B(1)} X.
  \end{align*}
Since $\beta_2, \beta_3 \geq \varepsilon$ and $\beta_1+\beta_2+\beta_3=1$, the right-hand side is $\ll X^{1/2}/\log^A X$ for every $A > 0$.
  
Thus, by dyadic splitting, it suffices to show that, for every
$T\in[X^\varepsilon,T_1]$ and
\[
\mathcal I_T=[T,\min\{2T,T_1\}],
\]
we have
\[
\int_{\mathcal I_T} |A_1(1/2+it)A_2(1/2+it)A_3(1/2+it)|\dd t\ll_{A, B}\frac{X^{1/2}}{\log^A X}.
  \]
Let $C = C(B)$ be as in Lemma~\ref{le:MVHuGM}(ii) and let
$\mathcal I_T=\mathcal{T}_1\sqcup\mathcal{T}_2$, where
$\mathcal{T}_1$ consists of those $t\in\mathcal I_T$ for which
  \[
    |A_1(1/2+it)| < \frac{X^{\beta_1/2}}{T} \log^C X
    \]
and $\mathcal{T}_2=\mathcal I_T\setminus\mathcal{T}_1$.

By the Cauchy-Schwarz inequality and the mean value theorem (see e.g.~\cite[Theorem 9.1]{IwaniecKowalski}) it follows that
\begin{align*}
    &\int_{\mathcal{T}_1} |A_1(1/2+it) A_2(1/2+it) A_3(1/2+it)| \dd t \\
    &\ll \frac{X^{\beta_1/2}}{T} \log^C X \left(\int_{\mathcal{T}_1} |A_2(1/2+it)|^2 \dd t \int_{\mathcal{T}_1} |A_3(1/2+it)|^2 \dd t\right)^{1/2} \\
    &\ll_B \frac{X^{\beta_1/2}}{T} \left((T+X^{\beta_2})(T+X^{\beta_3})\right)^{1/2} \log^{O_B(1)} X \\
    &\ll_B \left(X^{\beta_1/2} + \frac{X^{(\beta_1+\beta_2)/2}+X^{(\beta_1+\beta_3)/2}}{T_0^{1/2}} + \frac{X^{(\beta_1+\beta_2+\beta_3)/2}}{T_0}\right) \log^{O_B(1)} X \\
  &\ll_{A, B} \frac{X^{1/2}}{\log^A X}
  \end{align*}
  for every $A > 0$.
  
  Thus it suffices to show that
  \[
   \int_{\mathcal{T}_2} |A_1(1/2+it)A_2(1/2+it)A_3(1/2+it)|\dd t \ll_{A,B}\frac{X^{1/2}}{\log^A X}.
  \]
Write $T=X^\tau$ with $\tau\in[\varepsilon,\tau_1]$. Note that, for
$t\in\mathcal T_2$, one has
\[
  |A_1(1/2+it)|
  \geq \frac{X^{\beta_1/2}}{T}\log^C X.
\]
We deduce from Lemma~\ref{le:MVHuGM}(i) and (ii) that, for every
$\nu\geq0$,
\[
  \mathcal N_{A_1}(\nu;\mathcal T_2)
  \ll
  X^{\max\left\{0,
  \min\{\tau+2\beta_1(2\nu-1),
  \mathcal K(2\beta_1,\nu;\tau)\}\right\}+o(1)}.
\]
Besides, for $j \in \{2, 3\}$ and $\nu\geq0$, we have by Lemma~\ref{le:MVHuGM}
\[
  \mathcal N_{A_j}(\nu ; \mathcal{T}_2)\ll \min\{X^{\mathcal{K}(\beta_j,\nu;\tau)+o(1)}, X^{\mathcal{K}(2\beta_j,\nu;\tau)+o(1)}\}.
  \]
The desired estimate now follows from Proposition~\ref{prop:one-var-check} and Lemma~\ref{le:16/35comput}. This completes the proof.
\end{proof}

\begin{remark}\label{rem4.3}
The example $(\beta_1, \beta_2, \beta_3)=(1/5,2/5,2/5)$ is a limiting case which shows that one cannot take $\tau_{1} = 16/35$ in Lemma~\ref{le:16/35comput}: Set $\nu=2/7$ and observe that then
$$H_{1,\tau_{1}}\Big(\frac{2}{7}\Big)=\tau_{1}+2\beta_{1}(2\nu-1)=\tau_{1}-\frac{6\beta_{1}}{7}=2/7,$$
$$H_{2,\tau_{1}}\Big(\frac{2}{7}\Big)=\tau_{1}+\beta_{2}(2\nu-1)=\tau_{1}-\frac{3\beta_{2}}{7}=2/7,$$ 
$$H_{3,\tau_{1}}\Big(\frac{2}{7}\Big)=\tau_{1}+\beta_{3}(2\nu-1)=\tau_{1}-\frac{3\beta_{3}}{7}=2/7.$$
Hence, taking $t=2/7$ it follows that $P_{j,\tau_{1}}(2/7)\leq 2\beta_{j}/7$ for each $1\leq j\leq 3$, and thus
$$P_{1,\tau_{1}}\Big(\frac{2}{7}\Big)+P_{2,\tau_{1}}\Big(\frac{2}{7}\Big)+P_{3,\tau_{1}}\Big(\frac{2}{7}\Big)\leq \frac{2}{7}=t,$$ so~\eqref{eq:HBICompPjClaim} does not hold.

In more classical language, these computations correspond to the fact that with current understanding of the large values of Dirichlet polynomials, we cannot rule out the scenario that
\begin{align*}
 \left|\sum_{n\sim X^{\beta_j}} \frac{a_j(n)}{n^{1/2+it}}\right|\geq X^{\beta_j(1/2-2/7)}\qquad j=1,2,3
\end{align*}
for $\gg X^{2/7}$ one-separated points $t\in[-X^{\tau_1},X^{\tau_1}]$, where $a_1(n)=1$ and $|a_j(n)|\ll d_B(n)$ for $j=2,3$ and some fixed $B$. 

In order to obtain the exponent $19/35-\varepsilon$ in Theorems \ref{thm:intro-mobius}, \ref{thm1.3}, \ref{thm1.4} and \ref{thm1.5}, we would need to essentially have the bound
\[
\int_{X^\varepsilon}^{X^{\tau_1+\varepsilon}}  \left|\sum_{n \sim X^{1/5}} \frac{1}{n^{1/2+it}}\right|^5 \dd t \ll \frac{X^{1/2}}{\log^A X},
\]
and it is the above-mentioned scenario which prevents us from obtaining this.
\end{remark}

\section{A Baker--Harman--Pintz type trilinear estimate} \label{sec:BHP}

We now prove a trilinear estimate for Dirichlet polynomials in the spirit of the work of Baker, Harman, and Pintz (see~\cite[Lemma 7.3]{HarmanBook}).
\begin{lemma}[Baker--Harman--Pintz trilinear estimates] \label{lem:BHP-trilinear-estimates}
Let $\varepsilon > 0$ and $B \geq 1$. Let
\[
  1\leq T_1\leq X^{2/5-\varepsilon}
\]
and $\mathcal{T} \subseteq [-T_1, T_1]$.
Let $\beta_1, \beta_2, \beta_3 \in [\varepsilon, 1/2]$ be such that $\beta_2 \leq \beta_3 \leq \beta_1$ and $\beta_1 + \beta_2 + \beta_3 = 1$. For $j \in \{1, 2, 3\}$, let
\begin{align} \label{eq:upper-saving-5.1}
  A_j(s)=\sum_{X^{\beta_j} <n\leq2X^{\beta_j}} a_{j}(n) n^{-s},
\end{align}
where
\[
  |a_j(n)|\ll_B d_B(n)\log^B X
\]
and
\begin{equation} \label{eq:BHPsup}
  \sup_{|t|\leq T_1}|A_2(1/2+it)|
  \ll_A \frac{X^{\beta_2/2}}{\log^A X}
\end{equation}
for every $A \geq 1.$

Then
\begin{equation}\label{eq:bhp-trilinear-integral}
  \int_{\mathcal{T}} |A_1(1/2+it)A_2(1/2+it)A_3(1/2+it)|\dd t \ll_{A, B}\frac{X^{1/2}}{\log^A X}
\end{equation}
for every $A \geq 1.$
\end{lemma}

An important computational step in the proof of Lemma~\ref{lem:BHP-trilinear-estimates} is the following lemma which we will prove in Section~\ref{sec:Comput}.
\begin{lemma} \label{le:2/5comput}
  Let $\varepsilon, \varepsilon_0 > 0$, $\tau \in [\varepsilon, \frac{2}{5}-\varepsilon]$, and, for $\beta \geq \varepsilon$, define the function $H_{\beta, \tau} \colon \mathbb{R}_{\geq 0} \to [0, \tau]$ via
  \begin{align} \label{eq:Hbetatau5.2}
    H_{\beta, \tau}(\nu) &:= \min_{k \in \{1, 2\}} \mathcal{K}(k\beta, \nu; \tau)
  \end{align}
  and the function $P_{\beta, \tau} \colon [0, \tau] \to \mathbb{R}_{\geq 0}$ via
  \[
    P_{\beta, \tau}(t) = \inf\{\beta \nu \colon \nu \in [\varepsilon^6, \varepsilon^{-1}], H_{\beta, \tau}(\nu) \geq t\}.
  \]

Then, for every $\beta_1 \in [2/5, 1/2]$ and $\beta_2, \beta_3 \in [1/5,1/3]$ with $\beta_1+\beta_2+\beta_3=1$ and every $t \in [0, \tau]$, we have
 \begin{equation} \label{eq:BHPCompPjClaim}
   P_{\beta_1, \tau}(t)+P_{\beta_2, \tau}(t)+P_{\beta_3, \tau}(t) > t+\varepsilon^{10}.
 \end{equation}
\end{lemma}

\begin{proof}[Proof of Lemma~\ref{lem:BHP-trilinear-estimates} assuming Lemma~\ref{le:2/5comput}]

By the Cauchy--Schwarz inequality, the mean value theorem for Dirichlet polynomials (see e.g.~\cite[Theorem 9.1]{IwaniecKowalski}), and the pointwise hypothesis~\eqref{eq:BHPsup},
\begin{align*}
& \int_{\mathcal T\cap[-2X^\varepsilon,2X^\varepsilon]}
 |A_1(1/2+it)A_2(1/2+it)A_3(1/2+it)|\,\dd t \\
 &\leq \sup_{|t|\leq T_1}|A_2(1/2+it)|
 \left(\int_{-2X^\varepsilon}^{2X^\varepsilon}|A_1(1/2+it)|^2\,\dd t\right)^{1/2}
   \left(\int_{-2X^\varepsilon}^{2X^\varepsilon}|A_3(1/2+it)|^2\,\dd t\right)^{1/2} \\
  &\ll_{A,B}\frac{X^{1/2}}{\log^A X}
\end{align*}
for any $A \geq 1$.
By dyadic splitting and symmetry (using complex conjugates of coefficients), it is enough to prove~\eqref{eq:bhp-trilinear-integral} with $\mathcal{T}$ replaced by
\[
 \mathcal T_T=\mathcal T\cap[T,2T],\qquad [T,2T]\subseteq[X^\varepsilon,T_1].
\]

Suppose first that $\beta_1-\beta_3\leq1/5$. The Dirichlet polynomial Baker--Harman--Pintz parallelogram estimate in \cite[Lemma~7.3]{HarmanBook}, in the form used in the proof of \cite[Lemma~3.5(i)]{MSTT}, applies with $\theta=3/5$, since
\[
 |\beta_1-\beta_3|\leq\frac15=2\theta-1,
 \qquad 1-\beta_1-\beta_3=\beta_2\leq\frac13<\frac25=4\theta-2,
\]
and its pointwise hypothesis for the factor corresponding to $\beta_2$ is precisely~\eqref{eq:BHPsup}. This gives
\[
 \int_{\mathcal T_T}|A_1(1/2+it)A_2(1/2+it)A_3(1/2+it)|\,\dd t
 \ll_{A,B}\frac{X^{1/2}}{\log^A X}.
\]

It remains to consider $\beta_1-\beta_3>1/5$. Since $\beta_3=1-\beta_1-\beta_2$, the ordering of the $\beta_j$ shows that the closure of this region is the triangular region with $(\beta_1,\beta_2,\beta_3)\in\mathfrak T_\triangle$, where
\begin{equation}\label{eq:triangle-region-set}
  \mathfrak T_{\triangle}\coloneqq \operatorname{conv}\left\{
  \left(\frac7{15},\frac4{15},\frac4{15}\right),
  \left(\frac12,\frac14,\frac14\right),
  \left(\frac12,\frac15,\frac3{10}\right)
  \right\},
\end{equation}
where $\operatorname{conv}$ denotes the convex hull. In particular, this triangle is contained in the compact region
\begin{equation}\label{eq:blue-range}
  \beta_1\in\left[\frac25,\frac12\right],\qquad \beta_2,\beta_3\in\left[\frac15,\frac{3}{10}\right];
\end{equation}
see Figure~\ref{fig:typeII-region}.

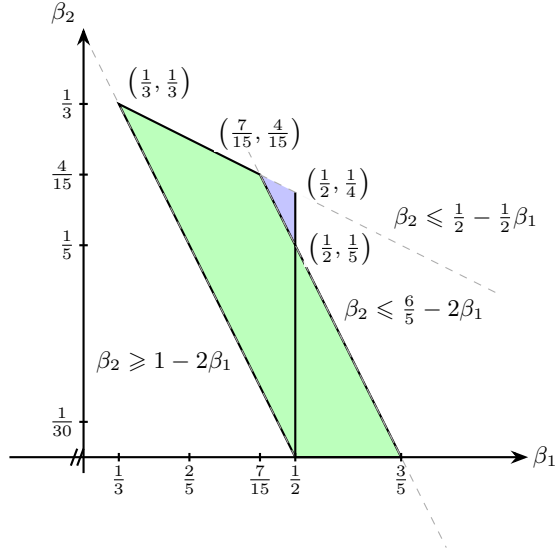
\begin{figure}[ht]
\centering
\begin{tikzpicture}[
  x=14cm,y=14cm,
  >=Stealth,
  font=\footnotesize,
  every node/.style={inner sep=1.4pt},
  labelnode/.style={fill=white,fill opacity=0.92,text opacity=1,inner sep=0.9pt,font=\scriptsize},
  ineqnode/.style={fill=white,fill opacity=0.96,text opacity=1,inner sep=1.8pt},
  guide/.style={thin,dashed,gray!65}
]
\coordinate (A) at ({1/3},{1/3});
\coordinate (B) at ({7/15},{4/15});
\coordinate (C) at ({3/5},0);
\coordinate (D) at ({1/2},0);
\coordinate (E) at ({1/2},{1/4});
\coordinate (F) at ({1/2},{1/5});
\coordinate (UpperRightExt) at ({0.69},{0.155});
\fill[green!35,opacity=0.75] (A)--(B)--(C)--(D)--cycle;
\fill[blue!35,opacity=0.65] (B)--(E)--(F)--cycle;
\draw[thick] (A)--(B)--(C)--(D)--cycle;
\draw[thick] (D)--(E);
\draw[guide] (B)--(UpperRightExt);
\draw[guide] (B)--($(B)!1.32!(C)$);
\draw[guide] (D)--($(D)!1.16!(A)$);
\draw[guide] (C)--($(C)!1.15!(B)$);
\draw[->,thick] (0.23,0)--(0.72,0) node[right] {$\beta_1$};
\draw[->,thick] (0.30,-0.015)--(0.30,0.405);
\node[anchor=south east] at (0.297,0.405) {$\beta_2$};
\draw[thick] (0.288,-0.008)--(0.294,0.008);
\draw[thick] (0.294,-0.008)--(0.300,0.008);
\draw[thick] ({1/3},0.004)--({1/3},-0.004);
\node[below] at ({1/3},-0.004) {$\frac13$};
\draw[thick] ({2/5},0.004)--({2/5},-0.004);
\node[below] at ({2/5},-0.004) {$\frac25$};
\draw[thick] ({7/15},0.004)--({7/15},-0.004);
\node[below] at ({7/15},-0.004) {$\frac7{15}$};
\draw[thick] ({1/2},0.004)--({1/2},-0.004);
\node[below] at ({1/2},-0.004) {$\frac12$};
\draw[thick] ({3/5},0.004)--({3/5},-0.004);
\node[below] at ({3/5},-0.004) {$\frac35$};
\draw[thick] (0.296,{1/30})--(0.304,{1/30});
\node[left] at (0.296,{1/30}) {$\frac1{30}$};
\draw[thick] (0.296,{1/5})--(0.304,{1/5});
\node[left] at (0.296,{1/5}) {$\frac15$};
\draw[thick] (0.296,{4/15})--(0.304,{4/15});
\node[left] at (0.296,{4/15}) {$\frac4{15}$};
\draw[thick] (0.296,{1/3})--(0.304,{1/3});
\node[left] at (0.296,{1/3}) {$\frac13$};
\node[labelnode,anchor=south west] at ($(A)+(0.004,0.002)$) {$\left(\frac13,\frac13\right)$};
\node[labelnode,anchor=south] at ($(B)+(0.000,0.021)$) {$\left(\frac7{15},\frac4{15}\right)$};
\node[labelnode,anchor=west] at ($(E)+(0.006,0.008)$) {$\left(\frac12,\frac14\right)$};
\node[labelnode,anchor=west] at ($(F)+(0.008,-0.003)$) {$\left(\frac12,\frac15\right)$};
\node[ineqnode,anchor=west] at (0.307,0.090) {$\beta_2 \geq 1-2\beta_1$};
\node[ineqnode,anchor=west] at (0.588,0.226) {$\beta_2 \leq \frac12-\frac12\beta_1$};
\node[ineqnode,anchor=west] at (0.541,0.139) {$\beta_2 \leq \frac65-2\beta_1$};
\end{tikzpicture}
\caption{Trilinear exponent region. The green region is covered by the Baker--Harman--Pintz parallelogram estimate; the triangular region is the remaining region.} \label{fig:typeII-region}
\end{figure}

Write $T=X^{\tau}$ with $\tau\in[\varepsilon,2/5-\varepsilon]$. By Lemma~\ref{le:MVHuGM}, for any $\nu\geq0$ and $j\in\{1,2,3\}$,
\[
  \mathcal N_{A_j}(\nu;\mathcal T_T)
  \ll X^{\min_{k\in\{1,2\}}\mathcal K(k\beta_j,\nu;\tau)+o(1)}.
\]
Now Proposition~\ref{prop:one-var-check} and Lemma~\ref{le:2/5comput} give the required estimate on $\mathcal T_T$. Summing the dyadic blocks completes the proof.
\end{proof}

\begin{remark}
By extending the proof of Lemma~\ref{le:2/5comput} to a wider region of $\beta_1, \beta_2, \beta_3$ (allowing larger $k$ in~\eqref{eq:Hbetatau5.2}) we could avoid appealing to the Baker--Harman--Pintz parallelogram estimate. However, now we emphasise the new region we needed to cover.
\end{remark}

\begin{remark}\label{rem5.3}
  The instance $(\beta_{1},\beta_{2},\beta_{3})=(1/2,1/4,1/4)$ is a limiting case and shows that one cannot take $\tau=2/5$ in Lemma ~\ref{le:2/5comput}. Indeed, set $\nu=3/10$ and observe that then
  \begin{align*}
    H_{\beta_{1}, \tau}(3/10)&=2\beta_{1}\nu=3/10, \\
    H_{\beta_{3}, \tau}(3/10) = H_{\beta_{2}, \tau}(3/10)&=\min\Big(\tau+\beta_{2}(2\nu-1),4\beta_{2}\nu\Big)=\min(3/10,3/10)=3/10
  \end{align*}
  Hence, taking $t=3/10$ it follows that $P_{\beta_{j}, \tau}(3/10)\leq 3\beta_{j}/10$ for each $1\leq j\leq 3$, and hence
$$P_{\beta_{1}, \tau}(3/10)+P_{\beta_{2}, \tau}(3/10)+P_{\beta_{3}, \tau}(3/10)\leq \frac{3}{10}=t,$$ in which case ~\eqref{eq:BHPCompPjClaim} does not hold.

In a more classical language, these computations correspond to the fact that with current understanding of the large values of Dirichlet polynomials, we cannot rule out the scenario that, for bounded coefficients $a_j(n)$,
\[
  \left|\sum_{n \sim X^{\beta_j}} \frac{a_j(n)}{n^{1/2+it}}\right| \geq X^{\beta_j(1/2-3/10)}
\]
for $\gg X^{3/10}$ one-separated points $t \in [-X^{\tau}, X^{\tau}]$. 

In order to obtain the exponent $3/5-\varepsilon$ in Theorem~\ref{thm:intro-uniformity}, we would need to essentially have, for $T = X^{O(1)}$, the bound
\[
\int_{X^\varepsilon}^{X^{\tau+\varepsilon}}  \left|\sum_{n \sim X^{1/4}} \frac{1}{n^{1/2+it+iT}}\right|^4 \dd t \ll \frac{X^{1/2}}{\log^A X},
\]
and it is the above-mentioned scenario which prevents us from obtaining this --- the coefficients $a_j(n) = n^{-iT}$ prevent us from using type I information once $T$ is a largish power of $X$.
\end{remark}

\section{Higher order uniformity in short intervals}\label{sec:uniformity}

The purpose of this section is to prove the nilsequence discorrelation estimates for $\mu$ and $\Lambda-\Lambda_X^\sharp$ given in Theorem \ref{thm:discorrelation}.  As stated before, we use the same nilsequence terminology as in~\cite[Section~2.3]{MSTT}.

Before deducing Theorem~\ref{thm:intro-uniformity} from Theorem~\ref{thm:discorrelation}, we state a result that allows us to reduce that theorem to the range $H\leq X^{1-\varepsilon}$ (cf. also~\cite[Section 7]{Teravainen-Quantitative}).

\begin{lemma}[Gluing Gowers norms on adjacent intervals]
\label{lem:gowers-gluing}
Let $s \in \mathbb{N}$. Let  $a \in \mathbb{R}$ and $H\geq 2L\geq 4$, and write $I=(a,a+H]$. Then, for every $f\colon\mathbb Z\to\mathbb C$,
\[
  \|f\|_{U^s(I)}\ll_s\max_{\substack{I'\subseteq I\text{ an interval}\\ L/2<|I'|\leq L}}
  \|f\|_{U^s(I')}.
\]
\end{lemma}

\begin{proof}
Set
\[
  m=\left\lceil\frac{H}{L}\right\rceil,\qquad \ell=\frac{H}{m}.
\]
Then $L/2<\ell\leq L$. Partition $I$ into the consecutive intervals
\[
  I_j=(a+j\ell,a+(j+1)\ell],\qquad 0\leq j<m,
\]
and write $f_j=f\1_{I_j}$.

For a tuple $\boldsymbol{i}=(i_\omega)_{\omega\in\{0,1\}^s}$, with each $i_\omega\in\{0,\ldots,m-1\}$, set
\[
  \Lambda_{\boldsymbol{i}} =\sum_{x,h_1,\ldots,h_s\in\mathbb Z} \prod_{\omega\in\{0,1\}^s}\mathcal C^{|\omega|}f_{i_\omega}(x+\omega_1h_1+\cdots+\omega_sh_s).
\]
By multilinearity, we have
\begin{equation} \label{eq:f1InormLambdai}
  \|f\1_I\|_{U^s(\mathbb Z)}^{2^s}=\sum_{\boldsymbol{i}}\Lambda_{\boldsymbol{i}}.
\end{equation}

We first count the tuples $\boldsymbol{i}$ for which
$\Lambda_{\boldsymbol{i}}$ can be nonzero. Suppose that such a tuple $\boldsymbol{i}$ is fixed and let $x,h_1,\ldots, h_s$ be such that $\prod_{\omega\in \{0,1\}^s}\mathcal{C}^{|\omega|}f_{i_{\omega}}(x+\omega_1h_1+\cdots+\omega_sh_s)\neq 0$. Writing, for $\omega \in \{0, 1\}^s$,
\[
  n_\omega=x+\omega_1h_1+\cdots+\omega_sh_s,
\]
we have
\[
  n_\omega=(1-|\omega|)n_{\boldsymbol{0}}+\sum_{r=1}^s\omega_r n_{\boldsymbol{e}_r}.
\]
The condition $n_\omega\in I_{i_\omega}$ gives
\[
  i_\omega=(1-|\omega|)i_{\boldsymbol{0}}+\sum_{r=1}^s\omega_r i_{\boldsymbol{e}_r}+O_s(1).
\]
Therefore, once
\[
  i_{\boldsymbol{0}},i_{\boldsymbol{e}_1},\ldots,i_{\boldsymbol{e}_s}
\]
have been chosen, there are only $O_s(1)$ possibilities for every
remaining $i_\omega$. Hence at most $O_s(m^{s+1})$
tuples $\boldsymbol{i}$ contribute.

For each such tuple, the Gowers--Cauchy--Schwarz inequality gives
\begin{equation} \label{eq:afterGCS}
  |\Lambda_{\boldsymbol{i}}|\leq \prod_{\omega\in\{0,1\}^s} \|f_{i_\omega}\|_{U^s(\mathbb Z)}.
\end{equation}

Writing
\begin{equation*}
  \eta=
  \max_{\substack{I'\subseteq I\text{ an interval}\\ L/2<|I'|\leq L}}
  \|f\|_{U^s(I')},
\end{equation*}
we have
\[
  \|f_j\|_{U^s(\mathbb Z)}\leq \eta\|\1_{I_j}\|_{U^s(\mathbb Z)}\ll_s \eta\ell^{(s+1)/2^s}.
\]
Combining this with~\eqref{eq:afterGCS} yields
\[
  |\Lambda_{\boldsymbol{i}}|\ll_s \eta^{2^s}\ell^{s+1}.
\]
Summing over the $O_s(m^{s+1})$ possible tuples $\boldsymbol{i}$ that contribute non-trivially to~\eqref{eq:f1InormLambdai}, we obtain
\begin{equation} \label{eq:F1Iest}
  \|f\1_I\|_{U^s(\mathbb Z)}^{2^s}\ll_s\eta^{2^s}m^{s+1}\ell^{s+1}=\eta^{2^s}H^{s+1}.
\end{equation}
Finally,
\begin{equation} \label{eq:1Iest}
  \|\1_I\|_{U^s(\mathbb Z)}^{2^s}\asymp_s H^{s+1}.
\end{equation}
Dividing~\eqref{eq:F1Iest} by~\eqref{eq:1Iest} and taking $2^s$-th roots proves the claim.
\end{proof}

\begin{proof}[Proof of Theorem~\ref{thm:intro-uniformity} assuming Theorem~\ref{thm:discorrelation}]
Let $\eta=\min\{\varepsilon/2,1/10\}$. For $H\leq X^{1-\eta}$, the result follows from Theorem~\ref{thm:discorrelation}, with $\eta$ in place of $\varepsilon$, by the inverse theorem for the Gowers norms and the transference argument in \cite[Proof of Theorem~1.5 in Section~9]{MSTT}, with just $5/8$ replaced by $3/5$.

Suppose now that $H>X^{1-\eta}$ and set $L=X^{1-\eta}/2$. By Lemma~\ref{lem:gowers-gluing}, it suffices to show that whenever $I'=(Y,Y+H']\subseteq(X,X+H]$ is an interval with
\[
  L/2<H'\leq L,
\]
we have
\begin{equation} \label{eq:Gowerslongishclaim}
   \|\Lambda-\Lambda_X^\sharp\|_{U^s(I')} = o_{s}(1).
 \end{equation}
 
Now $X\leq Y\leq2X$ and, for all sufficiently large $X$,
\[
  Y^{3/5+\eta}\leq H'\leq Y^{1-\eta}.
\]
The first part of the proof, applied with base point $Y$, therefore gives
\[
  \|\Lambda-\Lambda_Y^{\sharp}\|_{U^s(I')}=o_{s,\varepsilon}(1).
\]
Let $w=w(X)$ tend to infinity sufficiently slowly and let $\Lambda_w$ be as in \cite[Section~9]{MSTT}. The argument proving \cite[(9.17)]{MSTT}, using \cite[Lemma~9.6(i) and Remark~9.2]{MSTT}, is uniform in the sieve level permitted there. Applied with base point $Y$ to the sieve levels defining $\Lambda_X^\sharp$ and $\Lambda_Y^\sharp$, it gives
\[
  \|\Lambda_X^\sharp-\Lambda_w\|_{U^s(I')}+\|\Lambda_Y^\sharp-\Lambda_w\|_{U^s(I')}=o_{s,\varepsilon}(1)
\]
uniformly for such $I'$. The triangle inequality now gives~\eqref{eq:Gowerslongishclaim}, and thus Theorem~\ref{thm:intro-uniformity} follows.
\end{proof}

\subsection{Combinatorial decomposition}

The following elementary lemma is the combinatorial decomposition needed after applying Heath-Brown's identity.

\begin{lemma}[Type I, Type $\mathrm{I}_2$, Type II trichotomy]\label{lem:comb3/5}
Let $k\geq 1$ and let $\alpha_1,\ldots,\alpha_k\geq0$ with $\sum_{j\leq k}\alpha_j=1$. For $J\subseteq\{1,\ldots,k\}$ write $\alpha_J=\sum_{j\in J}\alpha_j$. At least one of the following alternatives holds.
\begin{enumerate}[label=(\roman*)]
\item $\alpha_j\geq2/5$ for some $j$.
\item $\alpha_j+\alpha_\ell\geq3/5$ for some distinct $j,\ell$.
\item There is a subset $J^*$ with $2/5\leq\alpha_{J^*}\leq3/5$, and there is a partition $\{1,\ldots,k\}=J_1\sqcup J_2\sqcup J_3$ such that $\alpha_{J_r}\in [1/10, 1/2]$ for $r=1,2,3$.
\end{enumerate}
\end{lemma}

\begin{proof}
  We assume that both \emph{(i)} and \emph{(ii)} fail and will show that case \emph{(iii)} must hold. We begin by finding a set $J^*$ such that $2/5\leq\alpha_{J^*}\leq3/5$. Without loss of generality, suppose $\alpha_1\geq \alpha_2\geq \cdots \geq\alpha_k\geq 0$.

  If $\alpha_2\geq 1/5$, then we set $J^* =\{1,2\}$ and the desired bound follows from $\alpha_1\geq \alpha_2$ and the assumption that $\alpha_1+\alpha_2<3/5$.

  Otherwise, let $j^*$ be the least index such that $\sum_{1\leq j\leq j^*} \alpha_j\geq 2/5$ and set $J^*=\{1,2,\dots,j^*\}$. Since we are not in case \emph{(i)}, we have $j^*\geq 2$, and hence $\alpha_{j^*} \leq \alpha_2 < 1/5$. Therefore,
\[
\frac{2}{5} \leq \alpha_{J^*} = \sum_{1\leq j<j^*} \alpha_j+ \alpha_{j^*} < \frac{2}{5} + \frac{1}{5} = \frac{3}{5}.
\]
Therefore, such a set $J^*$ exists in both cases.

Set $J_1=J^*$ if $\alpha_{J^*}\leq1/2$, and otherwise set $J_1=(J^*)^c$. Then $2/5\leq\alpha_{J_1}\leq1/2$ and $1/2\leq\alpha_{J_1^c}\leq3/5$. Moreover, $|J_1^c|\geq2$, since we are not in case \emph{(i)}. It remains to partition $J_1^c$ into two sets $J_2$ and $J_3$ satisfying the required bounds. First suppose that there exists $j'\in J_1^c$ such that $\alpha_{j'}\geq1/10$. Then set $J_2=\{j'\}$ and $J_3=J_1^c\backslash J_2$. In this case $1/10\leq\alpha_{J_2}<2/5$, since we are not in case $(i)$, and
\[
  1/10<\alpha_{J_3}=\alpha_{J_1^c}-\alpha_{J_2}\leq1/2.
\]
Otherwise, $\alpha_j<1/10$ for all $j\in J_1^c$. Choose $J_2\subset J_1^c$ minimally such that $\sum_{j\in J_2}\alpha_j\geq1/10$ and set $J_3=J_1^c\backslash J_2$. Then $\alpha_{J_2}<1/5$ and hence
\[
  3/10<\alpha_{J_3}=\alpha_{J_1^c}-\alpha_{J_2}\leq1/2.
\]
This proves that case \emph{(iii)} always occurs when $(i)$ and $(ii)$ fail.
\end{proof}

\begin{definition}\label{def:admissible-trilinear}
  We call $a \colon \mathbb{N} \to \mathbb{R}$ an \emph{admissible trilinear component} if there exist exponents $\gamma_1, \gamma_2, \gamma_3 \in [1/10, 1/2]$ and functions $\alpha_1, \alpha_2, \alpha_3 \colon \mathbb{N} \to \mathbb{R}$ such that the following conditions hold.
\begin{enumerate}[label=(\roman*)]
\item We have
\begin{equation}\label{eq:admissible-trilinear-exponents}
 \frac1{10}\leq\gamma_2\leq\gamma_3\leq\gamma_1\leq\frac12,
 \qquad \gamma_1+\gamma_2+\gamma_3=1.
\end{equation}
and
  \[
  a=\alpha_1*\alpha_2*\alpha_3.
\]
\item The functions $\alpha_j$ are such that, for some $Z\asymp X$, we have
\[
 \operatorname{supp}(\alpha_j)
 \subseteq[Z^{\gamma_j},2Z^{\gamma_j}].
\]
Moreover, each $\alpha_j$ has the form
\[
 \alpha_j(n)=\1_{I_j}(n)(b_{j,1}*\cdots*b_{j,\ell_j})(n),
 \qquad I_j\subseteq[Z^{\gamma_j},2Z^{\gamma_j}],
\]
where $I_j$ is an interval, $\ell_1+\ell_2+\ell_3\leq20$, every $b_{j,r}$ is one of
\[
 \1_J,\qquad (\log \cdot)\1_J ,\qquad \mu\1_J
\]
for an interval $J\subseteq(N,2N]$, at most one of all the elementary factors is logarithmically weighted, and a M\"obius factor has $N\ll Z^{1/10}$.
\end{enumerate}
\end{definition}

We now use Heath-Brown's identity to reduce the M\"obius and von Mangoldt functions to Type I, $\mathrm I_2$ and II sums, in the sense of~\cite[Definition~4.1]{MSTT}.

\begin{proposition}[Heath-Brown reduction]\label{prop:HB-reduction}
Each of the functions $\mu$ and $\Lambda$ admits a representation on $(X,X+H]$ as a linear combination of $O(\log^{O(1)} X)$ sums with coefficients satisfying $|a(n)|\ll d_{20}(n)\log X$ of the following forms.
\begin{enumerate}[label=(\alph*)]
\item \text{(Type I sums)}
We have $a=\alpha*\beta$, where
$$
  \operatorname{supp}(\alpha)\subset[1,X^{3/5}\log^{O(1)} X],
  \qquad
  \sum_{n\leq A}|\alpha(n)|^2\ll A\log^{O(1)} X
$$
for every $A \geq 1$, and the factor $\beta$ is a dyadic $1$-weighted or logarithmically weighted factor, and $\|\beta\|_{\TV(\mathbb N)}\ll\log^{O(1)} X$.
\item \text{(Type $I_2$ sums)}
We have $a=\alpha*\beta_1*\beta_2$, where
$$
  \operatorname{supp}(\alpha)
  \subset [1,X^{2/5}\log^{O(1)} X] \qquad
  \sum_{n\leq A}|\alpha(n)|^2\ll A\log^{O(1)} X,
$$
for every $A \geq 1$, and
$$
  \|\beta_1\|_{\TV(\mathbb N)} +\|\beta_2\|_{\TV(\mathbb N)} \ll\log^{O(1)} X.
$$
\item \text{(Type II sums)} We have $a=\alpha*\beta$, where $\textnormal{supp}(\alpha)\subseteq (M,2M]$, $\textnormal{supp}(\beta)\subseteq (N,2N]$ and
$$
 M,N\in  [X^{2/5}\log^{-O(1)} X, X^{3/5}\log^{O(1)} X],
$$
and moreover the coefficient sequences satisfy $\sum_{M < n\leq 2M}|\alpha(n)|^2\ll M\log^{O(1)} X$ and $\sum_{N < n \leq 2N}|\beta(n)|^{2k}\ll N\log^{O(1)} X$ for $k=1,2$.

Furthermore, every Type II component may be taken to be an admissible trilinear component.
\end{enumerate}

Furthermore, if we define
\[
 \Lambda_{X,\mathrm I}^{\sharp}(n)
 \coloneqq
 \frac{\Pcal(R)}{\varphi(\Pcal(R))}
 \sum_{\substack{d\leq X^{3/10}\\ d\mid(n,\Pcal(R))}}\mu(d),
\]
then this function has a Type I representation as above and
\begin{equation}\label{eq:cramer-type-I-error}
 \sum_{X<n\leq X+H}
 \left|\Lambda_X^\sharp(n)-\Lambda_{X,\mathrm I}^\sharp(n)\right|
 \ll H\exp\bigl(-c \log^{1/20} X\bigr).
\end{equation}
\end{proposition}

\begin{proof}
Apply the Heath-Brown identities for $\Lambda$ and $\mu$ in \cite[Lemma~2.16]{MSTT}, with $L=10$, and subdivide every factor dyadically. That lemma gives at most $20$ factors, each of type $\1_J$, $(\log \cdot)\1_J$, or $\mu\1_J$, with at most one logarithmic factor and with every M\"obius factor of length $\ll X^{1/10}$. Lemma~\ref{lem:comb3/5}, applied to the resulting support exponents and followed by a dyadic subdivision of each of the three grouped product variables, produces respectively the Type I alternative in Proposition~\ref{prop:HB-reduction}\textup{(a)}, the Type $\mathrm I_2$ alternative in Proposition~\ref{prop:HB-reduction}\textup{(b)}, and the admissible trilinear alternative in Definition~\ref{def:admissible-trilinear}, with the moment estimates following from the standard divisor bounds, while the assertion concerning $\Lambda_{X,\mathrm I}^{\sharp}$ and its error is \cite[(4.7)--(4.8)]{MSTT}.
\end{proof}

\subsection{Trilinear sums}

The new arithmetic input leading to our improvement is the following trilinear estimate.

\begin{proposition}[High frequency trilinear major arcs at the $3/5$ scale]\label{prop:typeII-major}
Let $A,B\geq1$ and $\eps>0$ be fixed. Let
$$
  X^{3/5+\eps}\leq H\leq X^{1-\eps},
$$
and let $a_n$ be any admissible trilinear component arising from Proposition~\ref{prop:HB-reduction}\textup{(c)}. There is a constant $C_A=C(A,B,\eps)>0$ such that, uniformly for
$$
 \frac{X}{H} \log^{C_A} X\leq |T|\leq X^B,
$$
one has
\begin{equation} \label{eq:MajorArc3/5claim}
  \left\lvert\sum_{X<n\leq X+H}a_n n^{iT}\right\rvert^{\,*} \ll_{A,B,\eps} \frac{H}{\log^A X}.
\end{equation}
\end{proposition}

\begin{proof}
Let
$P=\{b_0<n \leq b_0+L: n\equiv a \mod q\}\subseteq[X,X+H]$
be a progression on which the supremum in~\eqref{eq:starred-abs} is attained for~\eqref{eq:MajorArc3/5claim}. It suffices to show that
\begin{equation}\label{eq:typeII-major-progression}
\Big| \sum_{n \in P} a_n n^{iT}\Big|\ll_{A,B, \eps } \frac{H}{\log^A X}.
\end{equation}
By the divisor bound, $a_n \ll_\varepsilon X^{\varepsilon/4}$, so that
\begin{equation*}
\sum_{n \in P} |a_n| \ll_{A,B, \eps } X^{\varepsilon/4} \left(\frac{L}{q}+1\right).
\end{equation*}
This gives~\eqref{eq:typeII-major-progression} unless $L \geq H X^{-\varepsilon/2}$ and $q \leq X^{\varepsilon/2}$. 

Write $d=(a,q)$, $q=dq'$ and $a=da'$, so that $(a',q')=1$. Shiu's bound \cite[Theorem~1]{Shiu}, applied after $n=dm$, gives
\begin{align}\label{eq:shiu}
 \sum_{n\in P}|a_n|
 \ll \left(\frac{L}{d\varphi(q')}+1\right)d_{20}(d)\log^{O(1)} X \ll \frac{L}{q^{1/2}} \log^{O(1)} X.
\end{align}
It follows, on taking $A_1$ sufficiently large in terms of $A,B,\eps$, that \eqref{eq:typeII-major-progression} holds if $L<H\log^{-A_1} X$ or $q>\log^{A_1} X$. We may therefore assume that
\[
L\geq H\log^{-A_1} X \quad \text{ and }\quad q\leq \log^{A_1} X.
\]

By the orthogonality of Dirichlet characters, we have
\begin{align}
\sum_{ \substack{ b_0<n \leq b_0+L \\ n \equiv a \mod q}} a_n n^{iT} &= \sum_{ \substack{ b_0<n \leq b_0+L \\ d \mid n \\ \frac{n}{d} \equiv a' \mod {q'}}} a_{n} n^{iT}\nonumber\\
&= \frac{1}{ \varphi(q')} \sum_{\chi \mod{q'}} \overline{\chi(a')}\sum_{\substack{b_0 <n \leq b_0+L \\ d \mid n}} a_{n} \chi(n/d) n^{iT}. \label{eq:dm}
\end{align}

Write $u_0=\lfloor b_0\rfloor+1/2$ and $v_0=\lfloor b_0+L\rfloor+1/2$, and let $T_1 \coloneqq \frac{X}{H}\log^{B_1} X$ with $B_1$ sufficiently large. The truncated Perron formula then gives
\begin{align}\begin{aligned}\label{eq:Perron}
\sum_{n}a_n\1_{d\mid n}\chi(n/d)n^{iT}\1_{[u_0,v_0]}(n)&=\frac{1}{2\pi i}\int_{1/2-iT_1}^{1/2+iT_1}\frac{v_0^s-u_0^s}{s}\sum_{n}a_nn^{-s+iT}\1_{d\mid n}\chi(n/d)\dd s\\
\quad &+ O\left(\frac{X}{T_1}\log^{O(1)} X+X^{o(1)}\right),
\end{aligned}
\end{align}
where the error bound follows from Shiu's bound after a dyadic subdivision by distance from the half-integer endpoints. Note that
\[
\frac{v_0^{1/2+it}-u_0^{1/2+it}}{1/2+it} \ll \frac{L}{X^{1/2}}.
\]

Fix a representation $a=\alpha_1*\alpha_2*\alpha_3$ and parameters
$Z,\gamma_1,\gamma_2,\gamma_3$ as in
Definition~\ref{def:admissible-trilinear}.
For each factorisation $d_1d_2d_3=d$, set
\begin{align*}
 r_j&=\frac{d}{d_1\cdots d_j},\\
 \mathcal M_j(s)&=
 \sum_{\substack{d_j\mid n\\(n/d_j,r_j)=1}}
 \alpha_j(n)\chi(n/d_j)n^{-s+iT},
 \qquad 1\leq j\leq3.
\end{align*}
Note that the number of factorisations is $O(\log^{O_{A_1}(1)}X)$. Splitting the convolution variables successively according to
$d_1=(n_1,d)$ and $d_2=(n_2,d/d_1)$ yields
\begin{align*}
 \sum_{d\mid n}a_n\chi(n/d)n^{-s+iT}=
 \sum_{d_1d_2d_3=d}
 \mathcal M_1(s)\mathcal M_2(s)\mathcal M_3(s).
\end{align*}

The polynomials $\mathcal{M}_j(s)$ have supports on $[Z^{\gamma_j},2Z^{\gamma_j}]$ with the exponents in \eqref{eq:admissible-trilinear-exponents}. Thus, returning to~\eqref{eq:Perron}, we reduce matters to the normalised Dirichlet polynomial estimate, for some $A_2$ sufficiently large in terms of $A_1$, namely
\begin{equation}\label{eq:perron-reduction}
  \int_{-T_1}^{T_1}\abs{\mathcal M_1(1/2+it)\mathcal M_2(1/2+it)\mathcal M_3(1/2+it)}\dd t \ll_{A,B,\eps} \frac{Z^{1/2}}{\log^{A_2} X},
\end{equation}
uniformly for all Dirichlet characters $\chi\pmod{q'}$ and all divisor allocations. If $C_A>B_1+2$, then for $|t|\leq T_1$,
\[
 Z^{\eps/3}\leq |t-T|\leq Z^{B+1}.
\]

Recall that each $\alpha_j$ is a Dirichlet convolution of a bounded number of indicator sequences of intervals, or logarithmically weighted such indicators, or copies of M\"obius function. We open the elementary convolutions in the shortest factor $\mathcal M_2$. We distribute the divisibility condition among the elementary
variables by the same successive gcd decomposition and remove
the resulting coprimality conditions by M\"obius inversion.
This introduces only a polylogarithmic loss. One of the boundedly many factor polynomials has length $Z^{c+o(1)}$ for some $c=c(\eps)>0$; after the other variables are fixed, it remains an interval sum with coefficient $1$, $\mu$ or $\log$ and a character of polylogarithmic modulus. Applying e.g. \cite[Lemma~3.9(i),(ii)]{MSTT} with $X = 2Z^{B+1}$  (and partial summation for $\log$) to that factor, and the trivial bound to the others, gives
$$
  \sup_{t\in [-T_1, T_1]}\abs{\mathcal M_2(1/2+it)} \ll_{A,B,\eps} Z^{\gamma_2/2} \log^{-A'} X
$$
for every $A'>0$. Moreover, $T_1\leq Z^{2/5-\eps/3}$ for sufficiently large $X$, and the support and coefficient hypotheses of Lemma~\ref{lem:BHP-trilinear-estimates} hold. Hence that lemma applies with $Z$ in place of $X$. Taking $A'$ sufficiently large gives \eqref{eq:perron-reduction}. This completes the proof.
\end{proof}

\subsection{Proof of Theorem~\ref{thm:discorrelation}}

We first observe that the argument in \cite[Theorems~3.1 and~4.2]{MSTT} reduces the nilsequence discorrelation problem in the range $X^{3/5+\eps }\leq H \leq X^{1-\eps }$ to an abelian Type II obstruction.

\begin{proposition}[Nilsequence discorrelation up to an abelian
Type II obstruction]
\label{prop:non-abelian}
Let $\eps>0$ and $d,D,A\geq1$. Suppose that
\begin{equation*}
 2\leq X^{3/5+\eps}\leq H\leq X^{1-\eps}.
 \end{equation*}
Let $G/\Gamma$ be a filtered nilmanifold of degree at most $d$,
dimension at most $D$, and complexity at most $\log^A X$. Let $F\colon G/\Gamma\to\mathbb C$ have zero mean and satisfy
\begin{equation*}
 \|F\|_{\mathrm{Lip}}\leq\log^A X,
\end{equation*}
let $g:\mathbb Z\to G$ be polynomial, and let
\begin{equation*}
 f\in\{\Lambda-\Lambda_X^\sharp,\mu\}.
\end{equation*}
Then
\begin{equation}\label{eq:uniform-nonabelian-bound}
 \left|\sum_{X<n\leq X+H}f(n)F(g(n)\Gamma)\right|^*\ll_{d,D,A,\eps}\frac{H}{\log^A X},
\end{equation}
unless the following exceptional alternative holds.

There exist $A_1=C(d,D,A,\eps)>0$, an admissible trilinear component $h$
arising from Proposition~\ref{prop:HB-reduction}\textup{(c)} and a real number
\begin{equation*}
 |T|\leq (X/H)^{d+1}\log^{A_1} X
\end{equation*}
such that
\begin{align}
 \left|\sum_{X<n\leq X+H}f(n) n^{iT} \right|^*&\gg_{d,D,A,\eps}H\log^{-A_1} X,
 \label{eq:original-mult-correlation}\\
 \left|\sum_{X<n\leq X+H}h(n)n^{iT}\right|^*&\gg_{d,D,A,\eps}H\log^{-A_1} X.
 \label{eq:component-multiplicative-correlation}
\end{align}
\end{proposition}

\begin{proof}
We argue by induction on the dimension of $G/\Gamma$. The zero-dimensional case is immediate, since a function with mean zero on a one-point nilmanifold vanishes.

Choose $C_0=C_0(d,D,A,\eps)$ sufficiently large. On enlarging $C_0$, we may assume that it dominates all logarithmic powers in Proposition~\ref{prop:HB-reduction}. Moreover, we may assume that
\begin{equation}\label{eq:absolute-mass-f}
 \sum_{X<n\leq X+H}|f(n)|\ll H\log^{C_0} X.
\end{equation}
Indeed, for $f=\mu$ this is immediate, while for $f=\Lambda-\Lambda_X^\sharp$ it follows from Shiu's bound and the definition of $\Lambda_X^\sharp$. Let $A_1$ be sufficiently large in terms of $d,D,A,\eps,C_0$ to absorb any logarithmic losses that follow.

Choose $E>A+C_0+10$. By the Fourier decomposition of \cite[Proposition~2.9]{MSTT}, we have
\[
 F=\sum_{\xi\in\Xi}F_\xi+O(\log^{-E} X),\qquad |\Xi|+\max_{\xi\in\Xi}\|F_\xi\|_{\mathrm{Lip}}\ll \log^{O_{d,D,A,E}(1)} X,
\]
where $F_\xi$ oscillates with central frequency $\xi$. For each $\xi\in\Xi$, define
\[
 m_\xi=\int_{G/\Gamma}F_\xi,\qquad F_\xi^0=F_\xi-m_\xi.
\]
Then $F_\xi^0$ has mean zero; if $\xi\ne0$, Haar invariance gives $m_\xi=0$, while the zero-frequency component factors through the quotient by the centre. The error term is controlled by \eqref{eq:absolute-mass-f}, the constants $m_\xi$ by the major arc estimate of \cite[Theorem~3.1(i)]{MSTT}, and every component invariant under a positive-dimensional central subgroup by the induction hypothesis on the corresponding quotient. Consequently, after enlarging $A_1$, we may assume, unless the exceptional alternative has already occurred, that $\dim Z(G)=1$ and that $F$ oscillates with a non-zero central frequency.

By Proposition~\ref{prop:HB-reduction}, the function $f$ can be represented as a linear combination of $O(\log^{O(1)} X)$ functions $h$ of the following forms:
\begin{enumerate}[label=(\roman*)]
\item $h$ is a $(\log^{-A_1} X, X^{3/5+\eps/5})$ type I sum;
\item $h$ is a 	$(\log^{-A_1} X, X^{2/5+\eps/5})$ type I$_2$ sum;
\item $h$ is a $(\log^{-A_1} X, A_{II}^-,A_{II}^+)$ type II sum satisfying \begin{align} \label{eq:typeII-two-sided-range}
X^{2/5}\log^{-O(1)} X \ll A_{II}^{-}\leq A_{II}^{+}\ll X^{3/5+\varepsilon/5}.
\end{align}
\end{enumerate}
For $f=\Lambda-\Lambda_X^\sharp$, the $L^1$ error in the last sentence of Proposition~\ref{prop:HB-reduction} is smaller than every required logarithmic loss. If the claimed discorrelation estimate fails, we have
\begin{equation}\label{eq:uniform-nonabelian-bound2}
 \left|\sum_{X<n\leq X+H}f(n)F(g(n)\Gamma)\right|^*\gg \frac{H}{ \log^{A+1} X}.
 \end{equation}
Then the triangle inequality and the number of components give one function $h$ of one of the above types for which
\begin{align}\label{eq:component_correlation}
\left| \sum_{X<n \leq X+H} h(n) F(g(n)\Gamma)\right|^* \gg \frac{H}{\log^{A+O(1)} X}.
\end{align}
We apply \cite[Theorem~4.2]{MSTT} to this component. In the Type I and Type $I_2$ cases, and in the non-abelian Type II case, it gives one of the following alternatives:
\begin{enumerate}
\item [\textbf{Case 1:}] a non-trivial horizontal character $\eta:G\to \R/\Z$ of Lipschitz norm $O_{d,D}(\log^{O_{d,D,A,A_1}(1)} X)$ satisfies
\begin{align} \label{eq:horiz-charac}
\|\eta \circ g\|_{C^\infty [X,X+H]} \ll_{d,D} \log^{O_{A_1}(1)} X;	
\end{align}
\item[\textbf{Case 2:}] the component $h$ is Type II and the nilmanifold is abelian.
\end{enumerate}
The size alternative in \cite[Theorem~4.2(i),(ii),(iv)]{MSTT} cannot occur. Indeed, the corresponding quantities in the three cases satisfy
\[
 X^{3/5+\eps/5}=o(H),\qquad X^{1/3}(X^{2/5+\eps/5})^{2/3}=X^{3/5+2\eps/15}=o(H),
\]
and, by \eqref{eq:typeII-two-sided-range},
\[
 \max\{A_{II}^+,X/A_{II}^-\}\ll X^{3/5+\eps/5}\log^{O(1)} X=o(H).
\]

In Case~2, since $\dim Z(G)=1$, the group $G$ is one-dimensional. The single-frequency property therefore gives
\[
F(g(n)\Gamma)=b\,e(P(n))
\]
for some $b\in\mathbb C$ and some polynomial $P:\mathbb Z\to\mathbb R$ of degree at most $d$. The original failed correlation~\eqref{eq:uniform-nonabelian-bound2} and the component correlation~\eqref{eq:component_correlation} give now
\begin{align}
\left|\sum_{X<n\leq X+H}f(n)e(P(n))\right|^*&\gg_{d,D,A,\eps}H\log^{-A_1} X,
                                              \label{eq:original-poly-correlation}\\
  \left|\sum_{X<n\leq X+H}h(n)e(P(n))\right|^*&\gg_{d,D,A,\eps}H\log^{-A_1} X.
 \label{eq:component-poly-correlation}
\end{align}
After normalising by $b$, \cite[Theorem~4.2(iii)]{MSTT} gives an integer $q\ll \log^{O_{A_1}(1)} X$ and a number $ |T| \leq (X/H)^{d+1} \log^{O_{A_1}(1)} X$ such that
\begin{align}
    \|e(P(n))n^{-iT}\|_{\mathrm{TV}((X,X+H]\cap\mathbb Z;q)} &\ll_{d,D,A,\eps} \log^{O(A_1)} X,
 \label{eq:poly-twist}
\end{align}
By \cite[Lemma~2.2(iii)]{MSTT}, combining~\eqref{eq:original-poly-correlation} and \eqref{eq:component-poly-correlation} with \eqref{eq:poly-twist} and enlarging $A_1$ give~\eqref{eq:original-mult-correlation} and  \eqref{eq:component-multiplicative-correlation}. This is the exceptional alternative in the statement.

Therefore, it remains to consider the case in which \eqref{eq:horiz-charac} holds. As in \cite[Section~4.2]{MSTT}, we apply \cite[Lemma~2.11]{MSTT} to factor $g\in\Poly(\Z\to G)$ as $g=\varepsilon g'\gamma$, where $\eps$ is smooth, $\gamma$ is rational and $g'\in\Poly(\Z,\ker\eta)$. Freezing $\eps$ and $\gamma$ incurs an error of $O(H \log^{-A-1} X)$ and splits $(X,X+H]$ into $O(\log^{O_{A_1}(1)} X)$ arithmetic progressions $P_j$, as in \cite[Section~4.2]{MSTT}. On each $P_j$, we are therefore reduced to a nilsequence $F_j(g_j(n)\Gamma_j)$ on a proper rational subnilmanifold $G_j/\Gamma_j$ satisfying $\dim G_j<\dim G$.

With the constants in this factorisation fixed, choose $B'$ sufficiently large in terms of $A_1,A,d,D,\eps$ to absorb the number of progressions and all other logarithmic losses. If the induction hypothesis, applied to the mean-zero part of some $F_j$ on the original interval $(X,X+H]$, yields its exceptional alternative, then it gives \eqref{eq:original-mult-correlation} and \eqref{eq:component-multiplicative-correlation}, together with the required bound for $T$, with a possibly larger logarithmic exponent. We enlarge the final value of $A_1$ to that exponent. Otherwise, after subtracting the mean of $F_j$, the major arc estimate and the induction hypothesis, both applied with saving exponent $B'$, give
\[
\Big|\sum_{n\in P_j}f(n)F_j(g_j(n)\Gamma_j)\Big|^{*}\ll \frac{H}{\log^{B'} X}.
\]
Intersecting an arbitrary arithmetic progression with the $P_j$ and summing the resulting estimates gives
\[
\Big|\sum_{X<n\leq X+H}f(n)F(g(n)\Gamma)\Big|^{*}\ll \frac{H}{\log^A X},
\]
as claimed. This completes the induction and hence the proof.
\end{proof}

\begin{proof}[Proof of Theorem~\ref{thm:discorrelation}]
Let $f\in\{\mu,\Lambda-\Lambda_X^\sharp\}$. Subtracting the mean of
$F$, the constant part is bounded by \cite[Theorem~3.1(i)]{MSTT};
thus we may suppose that $F$ has mean zero. The Lipschitz norm of the new $F$ is at most $2/\delta$. Fix a sufficiently large
$M=M(A,d,D)$, and write
\[
 A'=\max\{A,M+1\}.
\]
If $\delta<\log^{-M} X$, the trivial estimate
\[
 \left|\sum_{X<n\leq X+H}
 f(n)\overline{F(g(n)\Gamma)}\right|^*
 \ll\delta^{-1}H\log X
\]
is at most $\delta^{-O_{d,D}(1)} H \log^{-A} X$ once the exponent
implicit in $O_{d,D}(1)$ is chosen sufficiently large. If
$\delta\geq\log^{-M} X$, the complexity of $G/\Gamma$ is at most
$\log^M X$. We apply Proposition~\ref{prop:non-abelian}
to $\overline F$ with $A'$ in place of $A$. We reach
\[
 \left|\sum_{X<n\leq X+H}
 f(n)\overline{F(g(n)\Gamma)}\right|^*
 \ll_{d, D, A, \varepsilon} \frac{H}{\log^{A'} X}\ll \frac{H}{\log^A X},
\]
unless the alternative of that proposition holds.

It remains to rule out the alternative of
Proposition~\ref{prop:non-abelian}. Let $h,T$ and $A_1$ be as
there. We apply Proposition~\ref{prop:typeII-major} with $A_1+10$ and $d+2$ in place of $A$ and $B$
and let $C_2=C_2(A_1+10,d+2,\varepsilon)$ be the constant denoted $C_A$ there.
If
\[
 \frac{X}{H} \log^{C_2} X\leq |T|
 \leq \left(\frac{X}{H}\right)^{d+1}\log^{A_1} X,
\]
then $|T|\leq X^{d+2}$ for all sufficiently large $X$, and
Proposition~\ref{prop:typeII-major} gives
\[
 \left|\sum_{X<n\leq X+H}h(n)n^{iT}\right|^*
 \ll \frac{H}{\log^{A_1+10} X},
\]
contradicting \eqref{eq:component-multiplicative-correlation}.

Suppose instead that
\[
 |T|<\frac{X}{H} \log^{C_2} X.
\]
Then
\begin{equation}\label{eq:variation-nit}
 \|n^{iT}\|_{\mathrm{TV}((X,X+H]\cap\mathbb Z)}
 \ll 1+  \sum_{\substack{X<n \leq X+H}}
 |(n+1)^{iT}-n^{iT}|\ll 1+ |T| \frac{H}{X}
 \ll\log^{C_2} X.
\end{equation}
Combining this with~\eqref{eq:original-mult-correlation} and
\cite[Lemma~2.2(iii)]{MSTT}, we obtain
\[
 \left|\sum_{X<n\leq X+H}f(n)\right|^*
 \gg \frac{H}{\log^{A_1+C_2} X}.
\]
The major arc estimate of \cite[Theorem~3.1(i)]{MSTT}, applied with
saving exponent $A_1+C_2+1$, contradicts this. Hence the alternative
in Proposition~\ref{prop:non-abelian} cannot occur, and the theorem
follows.
\end{proof}
\section{Sums of the M\"obius function over short intervals}\label{sec:mobius-short}
The aim of this section is to prove Theorem~\ref{thm:intro-mobius}. Let us first give the relevant combinatorial decomposition lemma.

\begin{lemma}[Type I, Type II, Type I/II trichotomy]\label{lem:combHBI}
Let $\tau \in (0, 1/2)$, let $0\leq\eta<1-2\tau$, and let $k\geq 1$. Let $\alpha_1,\ldots,\alpha_k\geq0$, write $S=\sum_{j\leq k}\alpha_j$, and suppose that $|S-1|\leq\eta$. For $J\subset\{1,\ldots,k\}$ write $\alpha_J=\sum_{j\in J}\alpha_j$. Then at least one of the following alternatives holds.
\begin{enumerate}[label=(\roman*)]
\item There is an index $j_0 \in \{1, \dotsc, k\}$ such that $\alpha_{j_0} > \tau$.
\item There is a partition $\{1, \dotsc, k\} = J_1 \sqcup J_2$ with $\alpha_{J_1}, \alpha_{J_2} \in [\tau, S-\tau].$
\item There is an index $j_0 \in \{1, \dotsc, k\}$ and a partition $\{1, \dotsc, k\} =\{j_0\} \sqcup J_1\sqcup J_2$ such that $S-2\tau<\alpha_{j_0}\leq\tau$ and, for $r \in \{1, 2\}$, we have
\[
 S-2\tau<\alpha_{J_r}< \tau.
\]
\end{enumerate}
\end{lemma}

\begin{proof}
  Assume that \emph{(i)} and \emph{(ii)} fail. After discarding zero terms, choose maximal disjoint sets $J_1,J_2\subseteq\{1,\dotsc,k\}$ such that $\alpha_{J_1},\alpha_{J_2}<\tau$. Write $J_0=\{1,\dotsc,k\}\setminus(J_1\cup J_2)$.

  Let us first show that $|J_0|=1$. Since $S\geq 2\tau$, we have $J_0\neq\emptyset$. Assume for contradiction that there exist distinct $r_1,r_2\in J_0$. By maximality,
$$
\alpha_{J_1 \cup \{r_1\}}, \alpha_{J_2 \cup \{r_2\}} \geq \tau.
$$
Since we are not in case \emph{(ii)}, neither sum lies in $[\tau,S-\tau]$, so both exceed $S-\tau$. But then
\[
S \geq \alpha_{J_1 \cup \{r_1\}} + \alpha_{J_2 \cup \{r_2\}} > 2(S-\tau) \geq S,
\]
which is a contradiction. Hence $|J_0|=1$, so $J_0=\{j_0\}$ for some $j_0\in \{1,\ldots, k\}$.

Since case \emph{(i)} fails, $\alpha_{j_0}\leq\tau$. Moreover, neither $J_1$ nor $J_2$ is empty, since otherwise $S<2\tau$. We have
\[
 \alpha_{j_0}=S-\alpha_{J_1}-\alpha_{J_2}>S-2\tau,
\]
and, for $\{r,s\}=\{1,2\}$,
\[
 \alpha_{J_r}=S-\alpha_{j_0}-\alpha_{J_s}>S-2\tau.
\]
Thus if $(i)$ and $(ii)$ fail, we are indeed in case \emph{(iii)}.
\end{proof}

We prove Theorem~\ref{thm:intro-mobius} following the overall strategy of Matom\"aki--Ter\"av\"ainen~\cite{MatomakiTeravainenMobius}. We spell out the reduction from the short interval M\"obius sum to a comparison of Dirichlet polynomial components, since this is where the Heath-Brown\,--\,Iwaniec type estimate is used.

\subsection{Ramar\'e extraction and comparison} \label{ssec:Ramare}

Choose parameters
\begin{equation}\label{eq:ramare-PQ}
  P\coloneqq \exp(\log^{2/3+\eta_0} X),\qquad Q\coloneqq X^{1/(\log\log X)^2},
\end{equation}
where $\eta_0>0$ is chosen sufficiently small in terms of $\eps$, and satisfies $\eta_0<\min\{1/40, \eps/2\}$. Employing a Ramar\'e-type identity \cite[Lemma~3.1]{MatomakiTeravainenMobius} with $2\eta_0$ in place of $\varepsilon$, we obtain, uniformly for $X^{2\eta_0} \leq H\leq X$,
\begin{equation}\label{eq:ramare-identity-section2}
  \sum_{X<n\leq X+H}\mu(n) =- \sum_{\substack{X<prn\leq X+H\\ P<p\leq Q\\ r\leq X^{\eta_0}}} a_r\mu(n) +O\left(H\frac{\log P}{\log Q}\right),
\end{equation}
with coefficients satisfying $|a_r|\leq d(r)$. With the choice \eqref{eq:ramare-PQ}, the error term in \eqref{eq:ramare-identity-section2} is $\ll H(\log X)^{-1/3+2\eta_0+o(1)}$, and is therefore $O(H \log^{-1/3+\eps} X)$ by the choice of $\eta_0$.

We next decompose the main term in
\eqref{eq:ramare-identity-section2}. Apply Heath-Brown's identity \cite[(13.38)]{IwaniecKowalski} to $\mu(n)$ with parameter $K=20$, and then decompose all variables dyadically. The identity has $2j-1$ variables in its $j$th summand, and after adjoining dummy variables equal to $1$, this expresses the main term as a linear combination of $O(\log^{2K+2} X)$ sums of the following form:
\begin{equation}\label{eq:dyadic-ramare-component}
  S_{\mathcal{I}} \coloneqq \sum_{\substack{pr n_1\cdots n_{2K-1} \in \mathcal{I} \\ p \in (P_1, 2P_1], \, r \in (R, 2R], n_j \in (N_j, 2N_j] \\ p \leq Q}} a_r\alpha_1(n_1)\cdots\alpha_{2K-1}(n_{2K-1}),
\end{equation}
where $P_1 \in [P, Q], \, R \in [1/2, X^{\eta_0}], \, N_j \geq 1/2$, and
\begin{equation}\label{eq:product-dyadic-lengths}
  P_1R N_1\cdots N_{2K-1}\asymp X.
\end{equation}
For $1\leq j\leq K-1$, the coefficient $\alpha_j(n)$ is an interval indicator (or a dummy factor supported at $1$, with $N_j = 1/2$); for $K\leq j\leq2K-1$, it is the M\"obius function restricted to an interval with $N_j\leq(2X)^{1/K}$ (or a dummy factor supported at $1$, with $N_j = 1/2$). In particular, every factor of length greater than $(2 X)^{1/K}$ is a $1$-weighted factor.

Finally, define 
\begin{equation}\label{eq:Y-long-interval}
Y\coloneqq
X\exp\bigl(-3 \log^{1/3} X\bigr).
\end{equation}
For each dyadic component, we compare the summation over
$(X,X+H]$ with the corresponding summation over
$(X,X+Y]$. The required comparison is provided by the following
proposition.

\begin{proposition}\label{prop:ramare-component-comparison}
Let $\eps>0$, $B_0>0$, and $X^{19/35+\eps}\leq H\leq Y$. Let $S_{\mathcal I}$ be any dyadic component of the form \eqref{eq:dyadic-ramare-component}. Then
\begin{equation}\label{eq:ramare-component-comparison}
  S_{(X,X+H]}=\frac{H}{Y}S_{(X,X+Y]}+O_{\eps,B_0}\left(\frac{H}{\log^{B_0} X}\right).
\end{equation}
The estimate is uniform in the dyadic parameters satisfying \eqref{eq:product-dyadic-lengths}.
\end{proposition}

\begin{proof}[Proof of Theorem~\ref{thm:intro-mobius} assuming Proposition~\ref{prop:ramare-component-comparison}]

If $H>Y$, the Vinogradov--Korobov estimate for the summatory M\"obius function gives
\[
 \left|\sum_{X<n\leq X+H}\mu(n)\right|\ll X\exp\bigl(-c\log^{3/5} X(\log\log X)^{-1/5}\bigr)
 \ll \frac{H}{\log^{100} X},
\]
so there is nothing to prove. We may therefore assume that $H\leq Y$.

Choose $B_0 = 100$ in Proposition~\ref{prop:ramare-component-comparison} to compensate for the $O(\log^{2K+2} X)$ dyadic components. Summing \eqref{eq:ramare-component-comparison} and recombining Heath-Brown's identity and~\eqref{eq:ramare-identity-section2} gives
\begin{equation}\label{eq:mobius-recombined-long}
  \sum_{X<n\leq X+H}\mu(n)
  =\frac{H}{Y}\sum_{X<n\leq X+Y}\mu(n)+O_{\eps}\left(\frac{H}{\log^{1/3-\eps} X}\right).
\end{equation}
For some $c>0$, the main term on the right-hand side is $\ll H\exp(-c\log^{3/5} X(\log\log X)^{-1/5})$ by the classical zero-free region for $\zeta(s)$, so Theorem~\ref{thm:intro-mobius} follows.
\end{proof}

\subsection{The Dirichlet polynomial estimate}

We conclude this section with the proof of Proposition~\ref{prop:ramare-component-comparison}.

\begin{proof}[Proof of Proposition~\ref{prop:ramare-component-comparison}]
Fix $A>B_0$ sufficiently large in terms of $K = 20$ to absorb the fixed logarithmic losses in this proof.

For $1\leq j\leq 2K-1$, write
\[
\beta_j = \max\left\{\frac{\log N_j}{\log X}, 0\right\},\qquad \beta_{2K} = \max\left\{\frac{\log R}{\log X}, 0\right\},\qquad
 \rho=\frac{\log P_1}{\log X},
\]
and set $S=\sum_{j=1}^{2K}\beta_j$. By \eqref{eq:product-dyadic-lengths},
\[
 S=1-\rho+O(1/\log X),\qquad 0\leq\rho\leq\frac{1}{(\log\log X)^2}+O(1/\log X).
\]
Thus, for $X$ sufficiently large, Lemma~\ref{lem:combHBI} applies to $\beta_1,\dotsc,\beta_{2K}$ with
\[
 \tau=\frac{16}{35}-\frac{\eps}{4},\qquad
 \eta=\frac{2}{(\log\log X)^2}.
\]

We split into cases.

\medskip\noindent\textbf{Case 1: Lemma~\ref{lem:combHBI}(i) holds.}
Now there is an index $j_0$ with $\beta_{j_0}>\tau.$ Since $\tau > 2/K$, $\alpha_{j_0}$ is an interval indicator. Write its supporting interval as $I_{j_0}=(u,v]$. Putting the sum over $n_{j_0}$ inside, in this case we obtain, for some $|b_m| \leq d_{100}(m)$,
\[
 S_{(X,X+H]} - \frac{H}{Y}S_{(X,X+Y]} = \sum_{m \leq 2X/N_{j_0}} b_m \left(\sum_{\substack{X/m < n_{j_0} \leq (X+H)/m \\ n_{j_0} \in (u, v]}} 1 - \frac{H}{Y} \sum_{\substack{X/m < n_{j_0} \leq (X+Y)/m \\ n_{j_0} \in (u, v]}} 1\right).
  \]
Now the difference in the parentheses is $O(1)$ unless
\[
 X<mu\leq X+Y\qquad\text{or}\qquad X<mv\leq X+Y.
\]
There are $O(Y/N_{j_0}+1)$ integers $m$ for which one of these holds, and for each of them the difference in the parentheses is $O(1+HN_{j_0}/X) = O(HN_{j_0}/X)$. Shiu's bound for the divisor bounded coefficients $b_m$ thus gives
\[
S_{(X,X+H]} - \frac{H}{Y}S_{(X,X+Y]} \ll \left(\frac{X}{N_{j_0}} + \frac{HY}{X} + \frac{HN_{j_0}}{X} \right) \log^{O(1)}X.
\]
Since $N_{j_0}\in (X^\tau, 2X/P]$, $H\geq X^{19/35+\eps}$, and $Y/X=\exp(-3\log^{1/3} X)$, this is $\ll_{A,\eps} H/\log^A X$ as claimed.

\medskip\noindent\textbf{Case 2: Lemma~\ref{lem:combHBI}(ii) or (iii) holds.}
By a standard reduction using Perron's formula (see e.g.~\cite[Lemma 7.2]{HarmanBook}), it is enough to prove that, for every $A \geq 1$,
\begin{equation}\label{eq:perron-target-section2}
  \int_{T_0}^{T_1}\left|P_1(1/2+it)N_{2K}(1/2+it)\prod_{j=1}^{2K-1} N_j(1/2+it)\right|\dd t\ll_A \frac{X^{1/2}}{\log^A X},
\end{equation}
where
\[
 T_0=\exp(\log^{1/3} X),\qquad T_1=\frac{X^{1+\eps/2}}{H} \leq X^{16/35-\eps/2},
\]
and
\[
 P_1(s)=\sum_{P_1<p\leq \min\{2P_1,Q\}}p^{-s},\quad N_{2K}(s)=\sum_{R<r\leq2R}a_r r^{-s},\quad N_j(s)=\sum_{N_j<n\leq2N_j}\alpha_j(n)n^{-s}
\]
for $1\leq j\leq 2K-1$.

\medskip\noindent\textbf{Case 2.1: Lemma~\ref{lem:combHBI}(ii) holds.}
Let $J_1\sqcup J_2=\{1,\dotsc,2K\}$ be the partition in the lemma and set, for $\ell \in \{1, 2\}$,
\[
A_\ell(s)=\prod_{j\in J_\ell}N_j(s),
\]
and
\begin{equation*}
 \beta_{J_\ell}\coloneqq\sum_{j\in J_\ell}\beta_j,
\end{equation*}
so that $A_\ell(s)$ has length $X^{\beta_{J_\ell}}$.

Alternative \emph{(ii)} of Lemma~\ref{lem:combHBI} gives
\begin{equation}\label{eq:case2-product-lengths}
 \tau\leq\beta_{J_1},\beta_{J_2}\leq S-\tau.
\end{equation}
Since $\tau=\frac{16}{35}-\frac{\eps}{4}$,
\begin{equation}\label{eq:case2-lengths-exceed-T}
T_1\leq X^{16/35-\eps/2}=X^{\tau-\eps/4} \leq \min\{X^{\beta_{J_1}}, X^{\beta_{J_2}}\}.
\end{equation}
 For every $B_0>0$, by the zero-free region estimate in the form used in \cite[Lemma~4.3]{MatomakiTeravainenMobius}, we have
\begin{equation}\label{P1}
	\sup_{T_0\leq|t|\leq T_1}|P_1(1/2+it)|\ll_{B_0} P_1^{1/2}\log^{-B_0} X.
\end{equation}

The Cauchy--Schwarz inequality, followed by the mean-value theorem for Dirichlet polynomials  (see e.g.~\cite[Theorem 9.1]{IwaniecKowalski}),~\eqref{eq:case2-lengths-exceed-T}, and~\eqref{eq:product-dyadic-lengths} yields
\begin{align*}
 &\int_{T_0}^{T_1}
 |P_1(1/2+it)A_1(1/2+it)A_2(1/2+it)|\,\dd t\\
 &\ll_{B_0} P_1^{1/2} \log^{-B_0} X\left(\int_{T_0}^{T_1}|A_1(1/2+it)|^2\,\dd t
 \right)^{1/2}\left(\int_{T_0}^{T_1}|A_2(1/2+it)|^2\,\dd t\right)^{1/2}\\
  &\ll P_1^{1/2}\log^{-B_0} X\left(T_1 + X^{\beta_{J_1}}\right)^{1/2}\left(T_1 + X^{\beta_{J_2}}\right)^{1/2} \log^{O(1)} X\\
 &\ll P_1^{1/2}(X^{\beta_{J_1}+\beta_{J_2}})^{1/2}\log^{-B_0+O(1)} X \ll \frac{X^{1/2}}{\log^{B_0-O(1)} X}
\end{align*}
Choosing $B_0$ sufficiently large therefore gives~\eqref{eq:perron-target-section2}.

\medskip\noindent\textbf{Case 2.2: Lemma~\ref{lem:combHBI}(iii) holds.}
Let $\{1,\dotsc,2K\}=\{j_0\}\sqcup J_1\sqcup J_2$ be the partition in the lemma, and set $\beta_{J_r}=\sum_{j\in J_r}\beta_j$. Then
\[
 S-2\tau<\beta_{j_0},\beta_{J_1},\beta_{J_2}\leq\tau.
\]
Moreover,
\[
 S-2\tau=\frac3{35}+\frac\eps2-\rho+O(1/\log X)>\frac1K
\]
for sufficiently large $X$. Hence $j_0\leq K-1$, so $N_{j_0}(s)$ is an interval-indicator factor. Absorb $P_1(s)$ into either of the other two blocks and subdivide the two product supports into $O_K(1)$ dyadic pieces. For a fixed piece, write the product as $L(s)M(s)N(s)$, with $L(s)=N_{j_0}(s)$, let $L_0,M_0,N_0$ be the lengths of the polynomials, and set $\widetilde X=L_0M_0N_0\asymp_KX$. Writing
\[
  L_0 = \widetilde{X}^{\gamma_1}, \quad   M_0 =\widetilde{X}^{\gamma_2}, \quad   N_0 = \widetilde{X}^{\gamma_3},
  \]
we have $\gamma_1, \gamma_2, \gamma_3 \in [\eps/8,16/35-\eps/8]$ with $\gamma_1 + \gamma_2 + \gamma_3 = 1$. Furthermore, for sufficiently large $X$,
\[
T_1\leq\widetilde X^{16/35-\eps/8}, \qquad T_0\geq\exp(\log^{1/4}\widetilde X).
\]
Lemma~\ref{lem:hbi-trilinear-estimates}, applied with $\widetilde X$ in place of $X$ and $\eps/8$ in place of $\eps$, now gives \eqref{eq:perron-target-section2}.
\end{proof}

\section{Computational lemmas} \label{sec:Comput}
In this section we prove Lemmas~\ref{le:16/35comput} and~\ref{le:2/5comput}. Note that each of these is a claim about positivity of a piecewise linear function (of the shape $P_1(t)+P_2(t)+P_3(t)-t-\varepsilon^{10}$). A brute force approach would be to check the inequality at all the possible break-points of these functions. Here we provide different proofs, based on looking at slopes of the piecewise linear functions. 

Before turning to the proofs of Lemmas~\ref{le:16/35comput} and~\ref{le:2/5comput}, for $\beta \geq \varepsilon$ and $t \in [0, \tau]$, we define and study the function
\begin{align} \label{eq:Rbetadef}
  \begin{aligned}
    R_{\beta, \tau}(t) &:= \beta \inf\{\nu \geq 0 \colon \mathcal{K}(\beta, \nu; \tau) \geq t\} \\
    &= \max\left\{0, 
    \min\Bigl\{\frac{t}{2}, 
    \max\Bigl\{\frac{t-\tau+2\beta}{6}, 
    \min\Bigl\{\frac{t-\tau+\frac{8}{5}\beta}{4}, \frac{t-\frac{\tau}{2}+\beta}{4}\Bigr\}, \frac{t-\tau+\beta}{2}\Bigr\}\Bigr\}\right\}.\end{aligned} \end{align}
Splitting into cases, we obtain the following explicit formulas.
  If $\varepsilon \leq \beta \leq \tau/2$, then
\begin{equation} \label{eq:Rbtaulow}
R_{\beta, \tau}(t)=
\begin{cases}
0, &\text{if $0\leq t\leq\tau-2\beta$,}\\
\frac{t-\tau+2\beta}{6}, &\text{if $\tau-2\beta\leq t\leq\tau-\frac{4}{5}\beta$,} \\
\frac{t-\tau+\frac{8}{5}\beta}{4}, &\text{if $\tau-\frac{4}{5}\beta\leq t\leq\tau-\frac25\beta$},\\
\frac{t-\tau+\beta}{2}, &\text{if $\tau-\frac{2}{5}\beta\leq t\leq\tau$}.
\end{cases}
\end{equation}
If $\tau/2\leq \beta\leq 5\tau/6$, then
\begin{align} \label{eq:Rbtaulowmid}
R_{\beta, \tau}(t)=
\begin{cases}
\frac{t}{2}, &\text{if $0\leq t\leq \beta-\frac\tau2$,}\\
\frac{t-\tau+2\beta}{6}, &\text{if $\beta-\frac\tau2\leq t\leq\tau-\frac45\beta$}, \\
\frac{t-\tau+\frac{8}{5}\beta}{4}, & \text{if $\tau-\frac{4}{5}\beta\leq t\leq\tau-\frac25\beta$},\\
\frac{t-\tau+\beta}{2}, & \text{if $\tau-\frac25\beta\leq t\leq\tau$}.
\end{cases}
\end{align}
If $5\tau/6 \leq \beta\leq\tau$, then
\begin{align}  \label{eq:Rbtauupmid}
R_{\beta, \tau}(t)=
\begin{cases}
\frac t2, &\text{if $0\leq t\leq \beta-\frac\tau2$,} \\
\frac{t-\frac{\tau}{2}+\beta}{4}, & \text{if $\beta-\frac\tau2\leq t\leq\frac{3\tau}{2}-\beta$}, \\
\frac{t-\tau+\beta}{2}, & \text{if $\frac{3\tau}{2}-\beta\leq t\leq\tau$}.
\end{cases}
\end{align}
Finally, if $\beta > \tau$, then
\begin{equation} \label{eq:Rbtauup}
  R_{\beta, \tau}(t) = \frac{t}{2}.
\end{equation}

\begin{proof}[Proof of Lemma~\ref{le:16/35comput}] 
First of all, we observe that $H_{j,\tau}(\nu)$ is nondecreasing in $\tau\in [\eps, \tau_1]$. As a consequence, $\{ \nu: H_{j,\tau}(\nu) \geq t\}\subseteq \{ \nu: H_{j,\tau_1} (\nu) \geq t\}$. From the definition of $P_{j,\tau}$ in Lemma~\ref{le:16/35comput} we thus have for $\tau\in [\eps, \tau_1]$ and $0\leq t \leq \tau$ that
\[
P_{j,\tau}(t) \geq P_{j,\tau_1}(t).
\] 
For convenience, set $P_j(t) =P_{j,\tau_1}(t)$ and $R_\beta(t) =R_{\beta,\tau_1}(t)$. It then suffices to prove that 
\begin{align} \label{eq:16/35-reduc-tau1}
F(t) := P_1(t) + P_2(t) + P_3(t) -t > \varepsilon^{10}
\end{align}
for all $0\leq t\leq \tau_1$.

Since $\nu$ takes values in the interval $[\eps^6,\eps^{-1}]$, we can conclude the above inequality when $t$ is rather small. Indeed, for all $t \in [0, \tau_1]$ one has
  \[
    P_1(t) + P_2(t) + P_3(t) \geq(\beta_1+\beta_2+\beta_3)\varepsilon^6 = \varepsilon^6.
    \]
Thus, for $t \leq \varepsilon^6/2$, we have
\[
F(t)\geq \frac{\varepsilon^6}{2}>\varepsilon^{10}.
\]

Let us assume below that $\eps^6/2< t\leq \tau_1$. From the definition of $R_\beta(t)$, it is immediate that if $\mathcal K(\beta_2,\nu;\tau_1) \geq t$ and $\mathcal K(2\beta_2,\nu;\tau_1)\geq t$, then $\beta_2\nu \geq R_{\beta_2}(t)$ and $2\beta_2\nu \geq R_{2\beta_2}(t)$. Thus it follows from the definitions of $P_2(t)$ in Lemma~\ref{le:16/35comput} and $R_{\beta}(t)$ in~\eqref{eq:Rbetadef} that
\begin{align}\label{eq:BHPPjlower-p2}
P_2(t) \geq \max\left\{R_{\beta_2}(t), \frac{1}{2} R_{2\beta_2}(t)\right\}.	
\end{align}
Similarly,
\begin{align} \label{eq:BHPPjlower}
 P_3(t) \geq \max\left\{R_{\beta_3}(t), \frac{1}{2} R_{2\beta_3}(t)\right\}, \qquad
  P_1(t) \geq \max\left\{\frac{t-\tau_1+2\beta_1}{4}, \frac{1}{2} R_{2\beta_1}(t)\right\}.
\end{align}
We shall prove~\eqref{eq:16/35-reduc-tau1} for $t \in [\eps^6/2,\tau_1]$ using these piecewise-linear lower
bounds.

We first make a preliminary reduction. By~\eqref{eq:BHPPjlower}, \eqref{eq:BHPPjlower-p2} and \eqref{eq:Rbetadef}, we have, for $\beta_2, \beta_3 \leq \tau_1$, 
\begin{align*}
  P_1(t) + P_2(t) + P_3(t) &\geq \frac{t-\tau_1+2\beta_1}{4} + \min\left\{\frac{t}{4}, \frac{t-\tau_1+2\beta_2}{4}\right\} + \frac{t-\tau_1+\beta_3}{2} \\
  &= \min\left\{t-\frac{3}{4}\tau_1+\frac{\beta_1+\beta_3}{2}, \frac{1}{2}+t-\tau_1\right\}.
\end{align*}
The second term in the minimum is $t+3/70+\varepsilon>t+\varepsilon^{10}$. The first term is strictly greater than $t+\eps^{10}$ if
\[
  \beta_1+\beta_3 > \frac{3}{2}\tau_1+2\varepsilon^{10} \iff \beta_2 < 1-\frac{3}{2}\tau_1-2\varepsilon^{10}.
  \]
  Thus, using also the symmetry between $\beta_2$ and $\beta_3$, we may restrict ourselves to the case
  \[
    \beta_2, \beta_3 \in \left[1-\frac{3}{2}\tau_1-2\varepsilon^{10}, \tau_1\right] \subseteq [11/35 +\eps,\tau_{1}].
    \]
In particular, after interchanging $\beta_2$ and $\beta_3$ if necessary, we assume from now on that
\[
  \frac{11}{35} +\eps \leq \beta_2\leq \beta_3 \leq \tau_1.
\]
Since $\eps \in (0, 1/100)$, we have $11/35+\eps \geq \tau_1/2$, so $\beta_2,\beta_3\geq \tau_1/2$ in this range. When $\beta \geq \tau_1 /2$ let
 \begin{align}\label{eq:end-slope}
 p(\beta) =\beta -\tau_1/2	
 \end{align}
denote the first breakpoint of~\eqref{eq:Rbtaulowmid} and~\eqref{eq:Rbtauupmid}, thus the end of the initial slope-$1/2$ segment. Similarly let
\begin{align}\label{eq:return-slope}
 r(\beta) = 
 \begin{cases}
   \tau_1 -2\beta/5, \qquad &\text{ if } \beta \leq 5\tau_1/6;\\
 	3\tau_1/2 -\beta, &\text{ if } 5\tau_1/6 \leq \beta \leq \tau_1,
 \end{cases}	
\end{align}
denote the last breakpoint of~\eqref{eq:Rbtaulow},~\eqref{eq:Rbtaulowmid} and~\eqref{eq:Rbtauupmid}, thus the beginning of the final slope-$1/2$ segment. Then $r(\beta)$ is a continuous and non-increasing function in $\beta$. Since $R_{\beta_2}(t) =R_{\beta_3} (t) =t/2$ whenever $0\leq t \leq p(\beta_2)\leq p(\beta_3)$, we then have
\[
F(t) \geq P_1(t) \geq \beta_1 \eps^6 \geq \eps^7>\eps^{10},
\]
as desired.

It therefore remains to consider the case $p(\beta_2) \leq t \leq \tau_1$. We divide the rest of the argument into two cases depending on whether $\beta_1\geq\tau_1/2$ or $\beta_1 < \tau_1/2$. 

\textbf{Case 1: $\beta_1\geq \tau_1/2$}. 
Define
\[
  F_{+}(t) = R_{\beta_2}(t) + R_{\beta_3}(t) + \frac{t-\tau_1+2\beta_1}{4} -t.
\]
From~\eqref{eq:BHPPjlower-p2}--\eqref{eq:BHPPjlower}, it suffices to show that $F_{+}(t) > \eps^{10}$ whenever $t \in [p(\beta_2),\tau_1]$. We shall show that $F_+(t)$ attains its minimum at $t = r(\beta_2)$.

When $t \in [p(\beta_2), r(\beta_2)]$, the slope of $R_{\beta_2}$ is at most $1/4$, and the slope of $R_{\beta_3}$ is at most $1/2$. Therefore the slope of $F_+(t)$ is at most $0$.
On the other hand, when $t\geq r(\beta_2) \geq r(\beta_3)$ both $R_{\beta_2}$ and $R_{\beta_3}$ are on their final branches (noting $r(\beta)$ is non-increasing), hence both of them have slope $1/2$. Thus in this range $F_+$ has slope
\[
1/2+1/2+1/4-1=1/4 \geq 0.
\]
Consequently, $F_{+}$ attains its minimum at $r(\beta_2)$, and the minimum is
\begin{align*}
  F_+(r(\beta_2)) &= \frac{r(\beta_2)-\tau_1+\beta_2}{2} + \frac{r(\beta_2)-\tau_1+\beta_3}{2}+\frac{r(\beta_2)-\tau_1+2\beta_1}{4}-r(\beta_2) \\
  &= \frac{r(\beta_2)}{4}-\frac{5}{4}\tau_1+\frac{1}{2}.
\end{align*}
If $\beta_2\leq 5\tau_1/6$, then $r(\beta_2) =\tau_1-2\beta_2/5$ and therefore,
\[
F_{+} (r(\beta_2)) =  -\tau_1 -\frac{\beta_2}{10} + \frac12
\geq \frac12-\frac{13\tau_1}{12}
>\eps^{10},
\]
If instead $\beta_2> 5\tau_1/6$, then $r(\beta_{2})=\frac{3\tau_{1}}{2}-\beta_{2}$. Using also $\beta_2\leq(1-\beta_1)/2\leq1/2-\tau_1/4$, we obtain
\[
F_{+} (r(\beta_2)) = \frac12 -\frac{7\tau_1}{8} - \frac{\beta_2}{4}
\geq\frac38-\frac{13\tau_1}{16}
>\eps^{10}.
\]
This proves~\eqref{eq:16/35-reduc-tau1} in Case~1.

\textbf{Case 2: $\beta_1 < \tau_1/2$}.
Define
\[
  F_{-} (t) = R_{\beta_2}(t) + R_{\beta_3}(t) + R_{2\beta_1}(t)/2 -t.
\]
For $t \in [p(\beta_2),r(\beta_2)]$, the slopes of $R_{\beta_2}$, $R_{\beta_3}$, and $R_{2\beta_1}/2$ are at most $1/4$, $1/2$, and $1/4$, respectively, so the slope of $F_-$ is $\leq 0$ in this interval. For $t \in [r(\beta_2),\tau_1]$, the first two slopes equal $1/2$ and the slope of $R_{2\beta_1}$ is non-negative, so the slope of $F_-$ is $\geq 0$ in this interval. Therefore $F_-$ attains its minimum at $r(\beta_2)$, and the minimum is
\begin{align*}
  F_{-} (r(\beta_2)) &= \frac{r(\beta_2)-\tau_1+\beta_2}{2} + \frac{r(\beta_2)-\tau_1+\beta_3}{2}+\frac{R_{2\beta_1} (r(\beta_2))}{2}-r(\beta_2)\\
  &= \frac12 -\frac{\beta_1}{2}-\tau_1 +\frac{R_{2\beta_1} (r(\beta_2))}{2}.
\end{align*}

Thus it suffices to show that
\begin{align}\label{eq:claim-beta1}
 R_{2\beta_1} (r(\beta_2)) > 2\tau_1+\beta_1-1+2\varepsilon^{10}
\end{align}
Consider this for a fixed $\beta_1 < \tau_1/2$. Now $R_{2\beta_1}(t)$ is increasing in $t$ while $r(\beta_2)$ is decreasing in $\beta_2$. Thus $R_{2\beta_1}(r(\beta_2))$ is minimal when $\beta_2$ is maximal. Since $\beta_2 \leq \beta_3$ and $\beta_1 + \beta_2 + \beta_3 = 1$, we have $\beta_2 \leq (1-\beta_1)/2$. Thus it suffices to show~\eqref{eq:claim-beta1} for $\beta_2 = (1-\beta_1)/2$. For $\beta_1 < \tau_1/2$, we have $(1-\beta_1)/2 \geq 5\tau_1/6$. Thus, by definition of $r(\beta)$ it suffices to show that
\begin{equation} \label{eq:R2bet1bound}
 H(\beta_1) := R_{2\beta_1}\left(\frac{3\tau_1}{2}+\frac{\beta_1-1}{2}\right) - (2\tau_1+\beta_1-1+2\varepsilon^{10}) > 0.
\end{equation} 
Looking carefully at~\eqref{eq:Rbtaulow},~\eqref{eq:Rbtaulowmid},~\eqref{eq:Rbtauupmid},~\eqref{eq:Rbtauup}, we see that the slope of $H(\beta_1)$ is $\leq 0$ when the argument of $R_{2\beta_1}$ is $< r(2\beta_1)$ and $\geq 0$ when the argument of $R_{2\beta_1}$ is $\geq r(2 \beta_1)$. Thus it suffices to consider the case
  \begin{equation} \label{eq:r2beta1eq}
    \frac{3\tau_1}{2}+\frac{\beta_1-1}{2} = r(2\beta_1)
  \end{equation}
  
    If we had $2 \beta_1 \leq 5\tau_1/6$, it would follow that
  \[
    r(2\beta_1) = \tau_1 - \frac{4\beta_1}{5}.
  \]
In this case
\[
  \frac{3\tau_1}{2}+\frac{\beta_1-1}{2} - r(2\beta_1) = \frac{\tau_1}{2} -\frac{1}{2} +\frac{13}{10}\beta_1 \leq \frac{25}{24}\tau_1 - \frac{1}{2} < 0.
  \]
Thus~\eqref{eq:r2beta1eq} does not have a solution with $2\beta_1 \leq 5\tau_1/6$ and $H(\beta_1)$ is non-increasing in this range.

For $2 \beta_1 > 5\tau_1/6$, by the definition of $r(\beta)$,
    \[
    \frac{3\tau_1}{2}+\frac{\beta_1-1}{2} = \frac{3\tau_1}{2} - 2\beta_1 \iff \beta_1 = \frac{1}{5}
  \]
  and
  \[
    R_{2/5}\left(\frac{3\tau_1}{2}+\frac{\frac{1}{5}-1}{2}\right) = \frac{\frac{3\tau_1}{2}-\frac{2}{5}-\tau_1+\frac{2}{5}}{2} = \frac{\tau_1}{4}. 
  \]
  Thus
  \[
    H(1/5) = \frac{\tau_1}{4}-(2\tau_1 + \frac{1}{5} -1 + 2\varepsilon^{10}) = -\frac{7\tau_1}{4} + \frac{4}{5} -2\varepsilon^{10}.
  \]
  Recalling $\tau_1 = 16/35-\varepsilon$, we see that $H(1/5) = \frac{7}{4}\varepsilon-2\varepsilon^{10}> 0$ and the claim follows.
  \end{proof}

\begin{proof}[Proof of Lemma~\ref{le:2/5comput}]

  Set $\tau^*=2/5-\eps$ and, for convenience, set $P_\beta(t) =P_{\beta,\tau^\ast}(t)$ and $R_\beta(t) =R_{\beta,\tau^\ast}(t)$. Arguing as in the beginning of the proof of Lemma~\ref{le:16/35comput}, it suffices to show that
  \begin{align} \label{eq:3/5-reduc-tau1}
F(t) := P_{\beta_1}(t) + P_{\beta_2}(t) + P_{\beta_3}(t) -t > \varepsilon^{10}
\end{align}
for all $t\in [\varepsilon^6/2, \tau^\ast]$. Note that, for all $\beta \in \{\beta_1, \beta_2, \beta_3\}$,
\begin{equation} \label{eq:Pbetatlow3/5}
P_\beta(t) \geq \max \Big\{R_{\beta} (t), \frac{R_{2\beta} (t)}{2}\Big\}.
\end{equation}
We may assume that $\beta_2\leq\beta_3$. Since
$\beta_2\geq1/5$ and $\beta_1\geq2/5$, we have
$\beta_1,2\beta_2\geq\tau^*$. We therefore deduce from~\eqref{eq:Pbetatlow3/5} and the definition
of $R_\beta(t)$ that 
\[
P_{\beta_1} (t) \geq R_{\beta_1}(t)=\frac{t}{2} \qquad \text{ and }\qquad P_{\beta_2}(t) \geq \frac{R_{2\beta_2} (t)}{2}=\frac{t}{4}.
\]
Therefore,
\[
P_{\beta_1}(t) + P_{\beta_2} (t) + P_{\beta_3} (t) -t \geq P_{\beta_3} (t) -\frac{t}{4}\geq R_{\beta_3}(t)-\frac{t}{4}.
\]
The remaining task is to prove that for all $t\in [\eps^6/2, \tau^*]$ we have
\begin{align} \label{eq:function-T}
R_{\beta_3}(t)-\frac{t}{4} > \eps^{10}.
\end{align}
First of all, we notice that
\begin{equation} \label{eq:beta3simplow}
  \beta_3\geq (1-\beta_1)/2 \geq 1/4\geq \tau^*/2.
\end{equation}
If $R_{\beta_3}(t)=t/2$, then~\eqref{eq:function-T} is immediate. In particular we can concentrate on the range $\beta_3 \in [1/4, \tau^\ast].$ Looking at the slopes of $R_{\beta}(t)$ in~\eqref{eq:Rbtaulowmid} and~\eqref{eq:Rbtauupmid}, we see that the left-hand side of~\eqref{eq:function-T} is minimal when $t = r(\beta_3)$ with
\begin{align}\label{eq:return-slope2/5}
 r(\beta) = 
 \begin{cases}
   \tau^\ast -2\beta/5, \qquad &\text{ if } \beta \leq 5\tau^\ast/6;\\
 	3\tau^\ast/2 -\beta, &\text{ if } 5\tau^\ast/6 \leq \beta \leq \tau^\ast,
 \end{cases}	
\end{align}
and
\[
  R_{\beta_3}(r(\beta_3))-\frac{r(\beta_3)}{4} = \frac{r(\beta_3)-\tau^\ast+\beta_3}{2}-\frac{r(\beta_3)}{4} = \frac{r(\beta_3)-2\tau^\ast+2\beta_3}{4}.
  \]
Now, for $\beta_3 \in [5\tau^\ast/6, \tau^\ast]$, we obtain
\[
  R_{\beta_3}(r(\beta_3))-\frac{r(\beta_3)}{4} = \frac{\beta_3 -\tau^\ast/2}{4} \geq \frac{\tau^\ast}{12} > \varepsilon^{10}.
\]
and for $\beta_3 \in [1/4, 5\tau^\ast/6]$, we obtain
\[
R_{\beta_3}(r(\beta_3))-\frac{r(\beta_3)}{4} =\frac{-\tau^* +8\beta_3/5}{4} \geq \frac{\eps}{4} >\eps^{10},
\]
which completes the proof.
\end{proof}

\section{Multiplicative functions and almost primes in short intervals} \label{sec:Mult}
In this section we shall briefly indicate how to prove Theorems \ref{thm1.3}, \ref{thm1.4} and \ref{thm1.5}. In case $H \in [X^{0.98}, X]$ we split the interval $(X, X+H]$ into intervals of length $X^{0.9}$ and apply the result with this $H$ for each subinterval. Hence we can assume that $H \in [X^{19/35+\varepsilon}, X^{0.98}]$.

We may assume that $\eps\in(0,1/100)$. The discussion in~\cite[Section~5.1]{MatomakiTeravainenMobius} reduces the proofs of Theorems~\ref{thm1.3} and~\ref{thm1.4} to comparison estimates of the form~\eqref{eq:ramare-component-comparison}. The resulting dyadic components contain a bounded number of factors with divisor-bounded coefficients, allowing fixed powers of
$\log X$. Every factor of length greater than $(2X)^{1/20}$ is
an interval sum weighted by a Dirichlet character of bounded
modulus, possibly with a logarithmic weight.

The proof of Proposition~\ref{prop:ramare-component-comparison}
extends to these components, since the extra short factors satisfy
the same divisor moment bounds, and $S-2\tau>1/20$ guarantees
that the singleton in case \textup{(iii)} is a
character-weighted interval factor. Case \textup{(i)} uses periodicity and partial summation, case \textup{(ii)} uses the character-weighted
prime polynomial estimate, and case \textup{(iii)} uses
Lemma~\ref{lem:hbi-trilinear-estimates} with a character weight. This yields Theorems~\ref{thm1.3} and~\ref{thm1.4}
whenever $X^{19/35+\eps}\leq H\leq Y$.

We now turn our attention to Theorem \ref{thm1.5}. Take $k=100$ and write $P=\exp((\log\log X)^{2})$ and $Q= X^{\varepsilon}$. Note that, as will become transparent later, we deliberately chose a different $k$ from the one taken in ~\cite[Section 5.2]{MatomakiTeravainenMobius}. Consider
$$P(s)=\sum_{P_{1}<p\leq \min \{2P_{1},Q\}}p^{-s},\ \ \ \ \ \ \ N_{i}(s)=\sum_{N_{i}<n\leq 2N_{i}} a_{i}(n)n^{-s},$$ where $P_{1}N_{1}\ldots N_{2k}\asymp X$, $P_1\in [P,Q]$, $N_i\geq 1/2$, and 
\[
a_i(n) \equiv
\begin{cases}
1 \ \text{ or } \log n \ \text{ or } 1_{n=1}, & \text{if } i \leq k, \\[4pt]
\mu(n)1_{n \leq (2X)^{1/k}}\  \text{ or } \ 1_{n=1}, & \text{if } k+1 \leq i \leq 2k.
\end{cases}
\]
 The discussion in ~\cite[Section 5.2]{MatomakiTeravainenMobius} leading to equation (5.7) reduces matters to showing that 
\begin{equation*}
\int_{T_{0}}^{T_{1}}\Big\lvert P\Big(\frac{1}{2}+it\Big) N_{1}\Big(\frac{1}{2}+it\Big) \ldots N_{2k}\Big(\frac{1}{2}+it\Big)\Big\rvert dt\ll_{A} \frac{X^{1/2}}{\log^A X},
\end{equation*}
with
\[
T_0 := X^{0.01}, \qquad T_1 := \frac{X^{1+\eps/10}}{H}.
\]
As was done in ~\cite{MatomakiTeravainenMobius}, we split the interval $[T_{0},T_{1}]$ into two parts, namely
$$\mathcal{T}_{1}=\{t\in [T_{0},T_{1}]: \lvert P(1/2+it)\rvert>P_{1}^{1/2} \log^{-10A} X\}$$ and its complement $\mathcal{T}_{2}=[T_{0},T_{1}]\setminus \mathcal{T}_{1}$. In order to estimate the contribution coming from $\mathcal{T}_{2}$ we follow the same procedure that was taken in the course of the proof of Proposition \ref{prop:ramare-component-comparison} to deduce ~\eqref{eq:perron-target-section2}, the only difference being that the bound $\lvert P(1/2+it)\rvert\leq P_{1}^{1/2}\log^{-10A} X$ stemming from the definition of $\mathcal{T}_{2}$ should replace the estimate ~\eqref{P1} utilised there.

For the integral \begin{equation}\label{PN1}
\int_{\mathcal{T}_{1}}\Big\lvert P\Big(\frac{1}{2}+it\Big) N_{1}\Big(\frac{1}{2}+it\Big) \ldots N_{2k}\Big(\frac{1}{2}+it\Big)\Big\rvert dt\ll_{A} \frac{X^{1/2}}{\log^A X},
\end{equation} 
we consider two cases. Suppose first that $N_{i}\leq (2X)^{1/k}$ for all $i\leq 2k$. Then, recalling that $k=100$ one may make a partition $\{1,\ldots 2k\}=\{j_{0}\}\sqcup J_{1}\sqcup J_{2}$ with $M_{1}:=\prod_{j\in J_{1}}N_{j}\in [X^{0.48},X^{1/2}]$, and $M_{2}:=\prod_{j\in J_{2}}N_{j}\in [X^{0.48},X^{0.52}]$ and $N_{j_{0}}\geq X^{1/1000}$. One then follows the procedure in the proof of Case 2.1 of Proposition \ref{prop:ramare-component-comparison} to get ~\eqref{PN1} by employing the pointwise bound $N_{j_{0}}(1/2+it)\ll N_{j_{0}}^{1/2}\log^{-A} X$ stemming from \cite[Lemma~3.9]{MSTT} and applying Cauchy--Schwarz to the Dirichlet polynomials $\prod_{j\in J_{1}}N_{j}(1/2+it)$ and $\prod_{j\in J_{2}}N_{j}(1/2+it)$. It is worth noting that in order to get the desired estimate, one should have
$$\frac{P_1^{1/2} N_{j_{0}}^{1/2}}{\log^A X}(T_{1}+M_{1})^{1/2}(T_{1}+M_{2})^{1/2}\ll \frac{X^{1/2}}{\log^A X}.$$
This holds in particular when $M_1, M_2 \geq T_1$ which is the case for us.

If instead there is some $j_{0}\leq 2k$ for which $N_{j_{0}}>(2X)^{1/k}$ then either $a_{j_{0}}(n)=1$ or $a_{j_{0}}(n)=\log n.$ Therefore, one can combine Weyl's method to estimate $N_{j_{0}}(1/2+it)$ and the bound $\lvert \mathcal{T}_{1}\rvert\ll X^{o(1)}$ to derive~\eqref{PN1} (see ~\cite[Section 5.2]{MatomakiTeravainenMobius} for further details).
\bibliographystyle{plain}
\bibliography{refs}

@article{MSTT,
  author       = {Matom{\"a}ki, K. and Shao, X. and Tao, T. and Ter{\"a}v{\"a}inen, J.},
  title        = {Higher uniformity of arithmetic functions in short intervals {I}. {All intervals}},
  journal      = {Forum Math. Pi},
  volume       = {11},
  pages        = {e29},
  year         = {2023},
  doi          = {10.1017/fmp.2023.28},
  eprint       = {2204.03754},
  archivePrefix= {arXiv},
  primaryClass = {math.NT},
  display-string = {(2023) Matomaki, K. and Shao, X. and Tao, T. and Teravainen, J. - Higher uniformity of arithmetic functions in short intervals I. All intervals}
}

@article{PW,
title={Discorrelation of the {M}\"obius function with linear phases in short intervals},
author={Pliego, J. and Wang, M.},
JOURNAL ={in preparation},
}

@book{IwaniecKowalski,
  author       = {Iwaniec, H. and Kowalski, E.},
  title        = {Analytic Number Theory},
  series       = {American Mathematical Society Colloquium Publications},
  volume       = {53},
  publisher    = {American Mathematical Society},
  year         = {2004},
  display-string = {(2004) Iwaniec, H. and Kowalski, E. - Analytic Number Theory}
}

@book{HarmanBook,
  author       = {Harman, G.},
  title        = {Prime-Detecting Sieves},
  series       = {London Mathematical Society Monographs},
  number       = {33},
  publisher    = {Princeton University Press},
  year         = {2007},
  display-string = {(2007) Harman, G. - Prime-Detecting Sieves}
}

@article {H-B-o1,
    AUTHOR = {Heath-Brown, D. R.},
     TITLE = {The number of primes in a short interval},
   JOURNAL = {J. Reine Angew. Math.},
  FJOURNAL = {Journal f\"ur die Reine und Angewandte Mathematik. [Crelle's
              Journal]},
    VOLUME = {389},
      YEAR = {1988},
     PAGES = {22--63},
      ISSN = {0075-4102,1435-5345},
   MRCLASS = {11N05 (11N35)},
  MRNUMBER = {953665},
MRREVIEWER = {R.\ C.\ Vaughan},
       DOI = {10.1515/crll.1988.389.22},
       URL = {https://doi.org/10.1515/crll.1988.389.22},
  display-string = {(1988) Heath-Brown, D. R. - The number of primes in a short interval},
}

@article{MatomakiTeravainenMobius,
  author       = {Matom{\"a}ki, K. and Ter{\"a}v{\"a}inen, J.},
  title        = {On the {M}{\"o}bius function in all short intervals},
  journal      = {J. Eur. Math. Soc. (JEMS)},
  volume       = {25},
  number       = {4},
  pages        = {1207--1225},
  year         = {2023},
  doi          = {10.4171/JEMS/1205},
  eprint       = {1911.09076},
  archivePrefix= {arXiv},
  primaryClass = {math.NT},
  display-string = {(2023) Matomaki, K. and Teravainen, J. - On the Mobius function in all short intervals}
}

@ARTICLE{CGCL,
       author = {{Chen}, B. and {Gupta}, V. and {Li}, Y. C.},
        title = "{Large Value Estimates for Dirichlet Polynomials with Characters and Zero Density of Dirichlet $L$-Functions}",
      journal = {arXiv e-prints},
         year = 2025,
        month = jul,
          eid = {arXiv:2507.08296},
        pages = {arXiv:2507.08296},
          doi = {10.48550/arXiv.2507.08296},
archivePrefix = {arXiv},
       eprint = {2507.08296},
 primaryClass = {math.NT},
       adsurl = {https://ui.adsabs.harvard.edu/abs/2025arXiv250708296C}
}

@incollection{BakerHarmanPintz1997,
  author       = {Baker, R. C. and Harman, G. and Pintz, J.},
  title        = {The exceptional set for {G}oldbach's problem in short intervals},
  booktitle    = {Sieve Methods, Exponential Sums, and their Applications in Number Theory},
  editor       = {Greaves, G. R. H. and Harman, G. and Huxley, M. N.},
  series       = {London Mathematical Society Lecture Note Series},
  volume       = {237},
  pages        = {1--54},
  publisher    = {Cambridge University Press},
  address      = {Cambridge},
  year         = {1997},
  doi          = {10.1017/CBO9780511526091.004},
  display-string = {(1997) Baker, R. C. and Harman, G. and Pintz, J. - The exceptional set for Goldbach's problem in short intervals}
}

@article {Ford02,
    AUTHOR = {Ford, K.},
     TITLE = {Vinogradov's integral and bounds for the {R}iemann zeta
              function},
   JOURNAL = {Proc. London Math. Soc. (3)},
  FJOURNAL = {Proceedings of the London Mathematical Society. Third Series},
    VOLUME = {85},
      YEAR = {2002},
    NUMBER = {3},
     PAGES = {565--633},
      ISSN = {0024-6115,1460-244X},
   MRCLASS = {11L15 (11D72 11M06 11M26 11M35 11Y35)},
  MRNUMBER = {1936814},
MRREVIEWER = {D.\ R.\ Heath-Brown},
       DOI = {10.1112/S0024611502013655},
       URL = {https://doi.org/10.1112/S0024611502013655},
  display-string = {(2002) Ford, Kevin - Vinogradov's integral and bounds for the Riemann zeta
              function},
}

@ARTICLE{MRII,
       author = {{Matom{\"a}ki}, K. and {Radziwi{\l}{\l}}, M.},
        title = "{Multiplicative functions in short intervals II}",
      journal = {arXiv e-prints},
         year = 2020,
        month = jul,
          eid = {arXiv:2007.04290},
        pages = {arXiv:2007.04290},
          doi = {10.48550/arXiv.2007.04290},
archivePrefix = {arXiv},
       eprint = {2007.04290},
 primaryClass = {math.NT},
       adsurl = {https://ui.adsabs.harvard.edu/abs/2020arXiv200704290M}
}

@article{GuthMaynard,
  author       = {Guth, L. and Maynard, J.},
  title        = {New large value estimates for {D}irichlet polynomials},
  journal      = {Ann. of Math. (2)},
  volume       = {203},
  number       = {2},
  pages        = {623--675},
  year         = {2026},
  doi          = {10.4007/annals.2026.203.2.6},
  eprint       = {2405.20552},
  archivePrefix= {arXiv},
  primaryClass = {math.NT},
  display-string = {(2026) Guth, L. and Maynard, J. - New large value estimates for Dirichlet polynomials}
}

@article{HeathBrownIwaniec1979,
  author       = {Heath-Brown, D. R. and Iwaniec, H.},
  title        = {On the difference between consecutive primes},
  journal      = {Invent. Math.},
  volume       = {55},
  pages        = {49--69},
  year         = {1979},
  doi          = {10.1007/BF01390094},
  display-string = {(1979) Heath-Brown, D. R. and Iwaniec, H. - On the difference between consecutive primes}
}

@article{Shiu,
  author  = {Shiu, P.},
  title   = {A {B}run--{T}itchmarsh theorem for multiplicative functions},
  journal = {J. Reine Angew. Math.},
  volume  = {313},
  pages   = {161--170},
  year    = {1980},
  display-string = {(1980) Shiu, P. - A Brun--Titchmarsh theorem for multiplicative functions}
}

@book {Tenenbaum2015,
	AUTHOR = {Tenenbaum, G\'erald},
	TITLE = {Introduction to analytic and probabilistic number theory},
	SERIES = {Graduate Studies in Mathematics},
	VOLUME = {163},
	EDITION = {Third},
	PUBLISHER = {American Mathematical Society, Providence, RI},
	YEAR = {2015},
	PAGES = {xxiv+629},
	ISBN = {978-0-8218-9854-3},
	MRCLASS = {11-02 (11Kxx 11Mxx 11Nxx)},
	MRNUMBER = {3363366},
	DOI = {10.1090/gsm/163},
	URL = {https://doi.org/10.1090/gsm/163},
}

@ARTICLE{Teravainen-Quantitative,
       author = {{Ter{\"a}v{\"a}inen}, J.},
        title = "{Quantitative Gowers uniformity of the primes in intervals of length $X^{5/8+\varepsilon}$}",
      journal = {arXiv e-prints},
         year = 2026,
        month = oct,
          eid = {arXiv:2610.03707},
        pages = {arXiv:2610.03707},
archivePrefix = {arXiv},
       eprint = {2610.03707},
 primaryClass = {math.NT},
       adsurl = {https://ui.adsabs.harvard.edu/abs/2026arXiv261003707T}
}

\end{document}